\documentclass[12pt]{amsart}
\usepackage{epsfig}

\usepackage{graphicx}

\newcommand{\R}{{\mathbb R}}
\newcommand{\C}{{\mathbb C}}

\newcommand{\Z}{{\mathbb Z}}
\newcommand{\N}{{\mathbb N}}

\newcommand{\lcal}{{\mathcal L}}

\newcommand{\acal}{{\mathcal A}}

\renewcommand{\phi}{\varphi}
\newcommand{\al}{\alpha}
\renewcommand{\k}{k}

\newcommand{\DB}{\Delta_{\Omega}^B}

\newcommand{\DL}{{\mathcal D}\ell(k)}
\newcommand{\SL}{{\mathcal S}\ell(k)}

\newcommand\Ha{\operatorname{Ha}^{(1)}}

\newcommand\csd{\Delta_{S^* \partial \Omega}}
\newcommand\Cbill{C_{\operatorname{billiard}}}

\newcommand\Omegabar{\overline{\Omega}}

\newcommand\pa{\partial}

\newtheorem{theo}{{\sc Theorem}}[section]
\newtheorem{cor}[theo]{{\sc Corollary}}

\newtheorem{lem}[theo]{{\sc Lemma}}

\newtheorem{prop}[theo]{{\sc Proposition}}
\newtheorem{defn}[theo]{{\sc Definition}}
\newtheorem{rem}[theo]{{\sc Remark}}

\newtheorem{maintheo}{{\sc Theorem}}
\newtheorem{maincor}{{\sc Corollary}}

\title[Inverse spectral problem for analytic $(\Z/2\,\Z)^n$-symmetric domains in $\R^n$]{Inverse spectral
 problem for analytic $(\Z/2\,\Z)^n$-symmetric domains in $\R^n$}

\author{Hamid Hezari and Steve Zelditch}
\address{Department of Mathematics, Johns Hopkins University, Baltimore, MD
21218, USA}

\email{zelditch@math.jhu.edu, hhezari@math.jhu.edu}
\thanks{Research partially supported by  NSF grant \#DMS-06-03850.}

\begin{document}

\maketitle

\begin{abstract}  We prove that bounded real analytic domains in
$\R^n$ with the symmetries of an ellipsoid, and with one axis
length fixed,  are determined by their Dirichlet or Neumann
eigenvalues among other bounded real analytic domains with the
same symmetries and axis length. Some non-degeneracy conditions
are also imposed on the class of domains. It follows that bounded,
convex analytic domains are determined by their spectra among
other such domains. This seems to be the first positive result on
the well-known Kac problem, can one hear the shape of a drum?, in
higher dimensions.

\end{abstract}

\section{Introduction and the Statement of results}

The purpose of this article is to prove that bounded analytic
domains $\Omega \subset \R^n$ with $\pm $ reflection symmetries
across all coordinate axes, and  with one axis height fixed  (and
also satisfying some generic non-degeneracy conditions) are
spectrally determined among other such domains. This inverse
result (Theorem \ref{ONESYM}) gives a higher dimensional analogue
of the main result of \cite{Z2} that ``bi-axisymmetric'' real
analytic plane domains are spectrally determined among other
bounded analytic plane domains with the symmetry of an ellipse. To
our knowledge, it is the first positive higher dimensional inverse
spectral result for Euclidean domains which is not restricted to
balls. Negative results (i.e. constructions of non-isometric
isospectral pairs) are given in \cite{U,GW,GWW} (see also
\cite{GS} for some non-Euclidean domains). Higher dimensional
inverse results for semi-classical Schr\"odinger operators with
similar symmetries have recently been proved in \cite{GU,H}.

\subsection{\label{STATEMENT} Statement of results}

We consider the eigenvalue problem on the domain $\Omega$ with the
Euclidean Laplacian $\DB$ and with boundary conditions $B$:

\begin{equation}\label{EIG}
\left\{ \begin{array}{l} \DB \phi_j(x)  = - \lambda_j^2
\phi_j(x),\;\;\;\langle \phi_i, \phi_j \rangle = \delta_{ij},\;\;
(x  \in \Omega) \\ \\
 B \phi_j (y) = 0, \;\;\; y \in
\partial \Omega.
\end{array}
\right.
\end{equation}
The  boundary conditions could be either Dirichlet $B \phi = \phi
|_{\partial \Omega}$,  or Neumann $B \phi = \partial_{\nu} \phi
|_{\partial \Omega}$ where $\partial_{\nu}$ is the interior unit
normal.

The  $(\Z/2\,\Z)^n$ symmetries of the title  are the maps
\begin{equation} \label{SIGMAJ} \sigma_j: (x_1, \dots, x_n) \to (x_1, \dots, - x_j, x_{j +
1}, \dots, x_n) \end{equation}  and we assume that they  are
isometries of $\Omega$.  The symmetry assumption implies that the
intersections of the coordinate axes with $\Omega$ are projections
of  bouncing ball orbits preserved by the symmetries.  We recall
that a bouncing ball orbit $\gamma$ is a 2-link periodic
trajectory of the billiard flow, i.e. a reversible periodic
billiard trajectory that bounces back and forth along a line
segment orthogonal to the boundary at both endpoints.  The
endpoints of the projection to $\Omega$ of the  bouncing ball
orbit are fixed points of all but one  of the isometries
$\sigma_j$; the remaining one fixes the projected orbit setwise
but interchanges the endpoints. We add the generic condition that
at least one of these bouncing ball orbits is
\textit{non-degenerate} (see (\ref{DL}) for the conditions). We
also fix the length $L_\gamma=2L$ of this bouncing ball orbit
$\gamma$.

We denote by $\mathcal D_{ L}$ to be the class of all bounded
real-analytic  domains $\Omega \subset \R^n$ satisfying these
assumptions. Thus, ${\mathcal D}_L$ is the class of domains such
that:

\begin{equation} \label{DL}
\left\{\begin{array}{llllll}
   \text{(i)}\; \; \;\sigma_j: \Omega \to \Omega \; \text{is an isometry for all}\; j=1,\dots, n; \\

   \text{(ii)} \;\text{one of the coordinate axis bouncing ball orbits, called $\gamma$, is  of length}\; 2L \;
             \;  \;\\

   \text{(iii)} \;\text{the lengths}\; 2 r L \; \text{of  all iterates}\; \gamma^r (r =
             1, 2, 3, \dots) \; \text{have multiplicity one in}\; Lsp(\Omega); \\

   \text{(iv)}\; \gamma \;\;  \text{is non-degenerate, i.e. 1 is not an eigenvalue of  its Poincar\'e map }  P_{\gamma}; \\
     \;\;\;\;\;\;\;\text{if}\; \gamma \;\text{is elliptic and}\; \{e^{\pm i\al_1},.
   ..e^{\pm i \al_{n-1}}\}\; \text{are the eigenvalues of } \; P_{\gamma}\; , \text{we} \\
              \; \; \; \;\; \;\;\text {further require that}\; \{\al_1,..., \al_{n-1}\} \;
              \text{are linearly independent over}\; \mathbb Q. \; \text{We assume the }\\ \;\;\;\;\;
              \;\;
              \text{same independence  condition in the Hyperbolic case or mixed cases.}\\
\end{array} \right.
\end{equation}

Here, $Lsp(\Omega)$ is the length spectrum of $\Omega$, i.e. the
set of lengths of closed billiard trajectories (cf. \cite{PS,Z3}).
Multiplicity one means that there exists precisely one closed
billiard trajectory of the given length up to time reversal. Let
Spec$_B(\Omega)$ denote the spectrum of the Laplacian
$\Delta_{\Omega}^B$  of the domain $\Omega$ with boundary
conditions $B$ (Dirichlet or Neumann).

\begin{maintheo} \label{ONESYM} For Dirichlet (or Neumann)  boundary
conditions $B$, the map Spec$_B: {\mathcal D}_{L} \to
\R_+^{{\bf N}}$ is 1-1. \end{maintheo}

In other words, if two bounded real analytic domains $\Omega_1,
\Omega_2 \subset \R^n$ possessing the symmetries of an ellipsoid
and satisfying the non-degeneracy and length  assumptions of
(\ref{DL}) have the same Dirichlet (resp. Neumann) spectra, then
they are isometric. To our knowledge, the  only prior positive
result on the inverse spectral problem for higher dimensional
bounded domains is that a domain with the Dirichlet (or Neumann)
spectrum of a ball must be a ball \cite{KAC}. In that case the
proof is based on the trace of the heat semi-group rather than the
wave group or resolvent kernel. The heat trace for the Dirichlet
(or Neumann) Laplacian $\Delta_{\Omega}^D$ of a bounded domain has
the singularity expansion,
$$Tr e^{ t \Delta_{\Omega}^D} \sim t^{-n/2} (C_n Vol_n(\Omega) + C_{n}'
Vol_{n-1}(\partial \Omega)  t^{-1/2} + \cdots ), \;\;\; t \to
0^+,$$ where $C_n, C_n'$ are constants depending only on the
dimension. Hence the volume and surface measure  are spectral
invariants. The ball is determined  as the unique domain where the
isometric inequality $Vol_{n-1} (\partial \Omega) \geq A_n
Vol_n(\Omega)^{\frac{n-1}{n}}$ (for a certain constant $A_n$) is
an equality.

Our proof of Theorem \ref{ONESYM} has a similar form in that we
calculate some special spectral invariants and then use the
invariants to uniquely determine the domain.   But instead of the
heat semi-group we use the wave group $e^{i t \sqrt{-
\Delta_{\Omega}^D}}$ or more precisely the semi-classical
resolvent $ R_{\Omega}^{D} (k) = -( \Delta_{\Omega}^D + k^2)^{-1}$ for
$k \in \C$, which is a semi-classical Laplace transform of the
wave group (see \S \ref{RWG}). Here,  we are assuming that the
boundary conditions are Dirichlet, but the methods and results are
valid for the Neumann Laplacian $\Delta_{\Omega}^N$ with only
minor modifications. The spectral invariants we study are the
`wave invariants' associated to one of the bouncing ball orbits
defined by the coordinate axes. The key advantage of these wave
invariants is that they are localized at the endpoints of the
projected orbit, whereas heat invariants are integrals of
curvature invariants over $\Omega$ or $\partial \Omega$. In
Theorem \ref{BGAMMAJ}, the wave invariants of bouncing ball orbits
are expressed in terms of the Taylor coefficients of the defining
function of $\Omega$ near the endpoints. Under our symmetry
assumptions, the Taylor coefficients are determined from the wave
invariants. That proves Theorem \ref{ONESYM}.

 As a corollary, we obtain a result for convex
analytic domains that does not require any length to be marked.

\begin{maincor} \label{ONESYMCOR} Let ${\mathcal C}$ be the class of analytic convex domains with
$(\Z/2\,\Z)^n$ symmetry, such that the shortest closed billiard
trajectory $\gamma_0$  is non-degenerate and satisfies the
conditions $(iii)$ and $(iv)$ of $(\ref{DL})$.  Then Spec$_B$:
${\mathcal C} \to \R_+^{{\bf N}}$ is 1-1.
\end{maincor}

This follows from Theorem \ref{ONESYM} and a result of M. Ghomi
\cite{Gh} that  the shortest closed trajectory of a
centrally-symmetric convex domain
 is automatically a bouncing ball orbit. Hence the length of this orbit is self-marked,
 and it is not necessary to mark the
  length $L_\gamma=2L$ of an invariant bouncing ball orbit $\gamma$.

  \subsection{Balian-Bloch and wave  invariants at a bouncing ball orbit}

 As mentioned above, the  proof of Theorem \ref{ONESYMCOR} is based the study of spectral invariants
  of the Dirichlet or Neumann Laplacian of $\Omega$
   known as the Balin-Bloch (or  wave trace) invariants at the closed billiard
   trajectories $\gamma$ of $\Omega$.    The Balian-Bloch
  invariants $B_{\gamma, j}$ are coefficients
  of  the regularized trace expansion
  \begin{equation} \label{PR} Tr R_{\Omega, \rho}^{D} (k ) \sim  {\mathcal D}_{D, \gamma}(k)  \sum_{j =
0}^{\infty} B_{\gamma, j}  k^{-j},\;\;\; \Re k \to
\infty,\end{equation}
 of the smoothed semi-classical resolvent $ R_{\Omega, \rho}^{D} (k)$ where $\hat{\rho}$ is localized
 at the length of the closed orbit $\gamma$. The smoothed semi-classical resolvent is defined
 in (\ref{RR}) and the  precise statement of (\ref{PR}) is given in Theorem \ref{BBL}.
   The factor ${\mathcal D}_{D, \gamma}(k)$ is a well-known
 symplectic factor that is  reviewed in \S \ref{RWG} and discussed
 in more detail in \cite{GM,PS}.

 The
 Balian-Bloch invariants are named after the physicists who introduced them in   \cite{BB1,BB2}
 and studied them on a somewhat non-rigorous  formal level.  Since then,  a long stream of
 mathematical works have been produced  on the
  dual singularity expansion of the trace $Tr \cos t \sqrt{-\Delta_{\Omega}^D}$  of the wave
 group. The classical results on wave trace invariants on compact Riemannian
 manifolds without boundary are due
to  Colin de
 Verdi\`ere, Chazarain and Duistermaat-Guillemin. The wave trace expansion
 was  then  generalized to  manifolds with
boundary by Guillemin-Melrose in \cite{GM} (see also  \cite{PS}
for a very thorough study).  The semi-classical
 resolvent and wave group are related by a Laplace transform (see (\ref{LAPLACE})) and so
 the semi-classical (i.e. large $k$)
 expansion (\ref{PR}) is
essentially the same as the singularity expansion of $Tr \cos t
\sqrt{-\Delta_{\Omega}^D}$.  Algorithms for calculating the
 coefficients in the boundaryless case were given in \cite{G,G2,Z1}.  In
 \cite{Z2,Z3} an algorithm was given for calculating the invariants in the
 boundary case but it was  only implemented for plane domains. In this
 article, the algorithm is developed for higher dimensional
 domains and explicit formulae for the Balian-Bloch or wave trace
 invariants are given in Theorem \ref{BGAMMAJ}. The calculations
 also draw on the analysis in \cite{H} of similar invariants for
 semi-classical Schr\"odinger operators. This result is of
 independent interest and is valid without any symmetry
 assumptions.

 We now  state the formulae for the Balian-Bloch invariants.  They require
 some more notation which will be further discussed in \S
 \ref{BACKGROUND} and \S
 \ref{SPSECT}. Almost the same notation is used in \cite{Z3}.
  We align the axes so that the bouncing ball orbit
$\gamma$ is  a vertical segment of length $L$  with endpoints at
$A = (0,\frac{L}{2})$ and $B = (0, -\frac{L}{2})$, where $0$
denotes the origin in the orthogonal $x'=(x^1, \dots, x^{n-1}) \in
\R^{n-1}$ plane. In a metric tube $T_{\epsilon}(\overline{AB})$ of
radius $\epsilon$ around $\gamma$, we may locally express
 $\partial \Omega = \partial \Omega^+ \cup \partial \Omega^-$ as the union of two graphs
over a ball $B_{\epsilon}(0)$ around $0$ in the $x'$-hyperplane,
namely
\begin{equation}\label{GRAPHS}  \partial \Omega^+
 = \{ x^n = f_+(x'),\;\;\; |x'| \leq \epsilon\},\;\; \partial \Omega^-
  = \{ x^n = f_-(x'),\;\;\; |x'| \leq \epsilon \}. \end{equation}

We will use the standard shorthand notations for multi-indices,
i.e. $\vec \gamma=(\gamma_1,...\gamma_{n-1})$, $|\vec
\gamma|=\gamma_1+...+\gamma_n$,  $\vec X ^{\vec
\gamma}=X_1^{\gamma_1}...X_{n-1}^{\gamma_{n-1}}$, and by
$\overrightarrow {h^{pq}_{\pm, 2r}}$ , we mean the $(n-1)$-vector
$$\overrightarrow {h^{pq}_{\pm, 2r}} = (h^{11,pq}_{\pm, 2r},h^{22,pq}_{\pm, 2r},...,h^{(n-1,n-1),pq}_{\pm,
2r}), $$ where $[h^{ij,pq}_{\pm, 2r}]_{1\leq i,j\leq n-1, 1\leq p, q
\leq 2r}$ is the inverse Hessian matrix of the length functional $
{\mathcal L}_{\pm}(x_1', \dots, x_{2r}')$ given in (\ref{LPM}).

The following generalizes Theorem 5.1 in \cite{Z3} from two to
higher dimensions.

\begin{maintheo} \label{BGAMMAJ} Let $\Omega$ be a smooth domain with a bouncing
ball orbit $\gamma$ of length $2L$ and let $B_{\gamma^r, j}$ be
the wave invariants associated to $\gamma^r$ (see cf. \ref{BB}).
Then for each $r = 1, 2, \dots$, and $j$ there exists a polynomial
$P_{r,j}$ such that:

\begin{enumerate}

\item $B_{\gamma^r, j}   = P_{r, j} (\{D^{|\gamma|}_{\vec \gamma}
f_+(0)\},\{D^{|\gamma|}_{\vec \gamma} f_-(0)\})$ with $|\gamma|
\leq 2j + 2$; i.e. the highest order of derivatives appearing in
$B_{\gamma^r, j}$ is $2j+2$.

\item   In the polynomial expansion of $B_{\gamma^r, j}$ the
Taylor coefficients of order $2j+2$ appear in the form $\{D^{2j+2}_{2\vec \gamma}
f_{\pm}(0)\}$.

\item $B_{\gamma^r,0}$ is only a function of $r$, $L$ and $n$, and
for $j\geq 1$
$$ \begin{array}{l}  B_{\gamma^r, j} =\frac{B_{\gamma^r,0}}{(2i)^{j+1}} \sum_{|\gamma|=j+1}\frac{r}{\vec \gamma !}\big\{(\overrightarrow {h^{11}_{+,
2r}})^{\vec \gamma}D^{2j+2}_{2\vec \gamma} f_{+}(0)
 -(\overrightarrow {h^{11}_{-,
2r}})^{\vec \gamma}D^{2j+2}_{2\vec \gamma} f_{-}(0)\big\}
\\ \\
 + R_{2r,j} ({\mathcal
J}^{2j+1} f_+(0), {\mathcal J}^{2j+1} f_-(0)),\end{array} $$ where
the remainder $ R_{2r,j} ({\mathcal J}^{2j+1} f_+(0), {\mathcal
J}^{2j+1} f_-(0))$ is a polynomial in the designated jet of
$f_{\pm}.$

\item In the $(\Z/2\,\Z)$- symmetric case, where $f_+=f=-f_-$, we
have the simplified formula
$$ \begin{array}{l}  B_{\gamma^r, j} = \frac{B_{\gamma^r,0}}{(2i)^{j+1}} \sum_{|\vec \gamma|=j+1}\frac{r}{\vec \gamma !}\left( \frac{1}{\sin  \frac{\vec \alpha}{2}} \cot \frac{r
\vec \alpha}{2} \right)^{\vec \gamma}D^{2j+2}_{2\vec \gamma} f(0)
  + R_{2r,j} ({\mathcal
J}^{2j+1} f(0)). \end{array}$$

\item In the $(\Z/2\,\Z)^n$- symmetric case, formula (4) holds
with remainder in $R_{2r,j} ({\mathcal J}^{2j} f(0)).$
\end{enumerate}
\end{maintheo}

In the above notation, the $(\Z/2\,\Z)^n$-symmetry assumptions in
(5)  are that

  \begin{equation} \label{SYMMETRY} f_+(x') = - f_-(x'), \;\; f_{\pm}(\sigma_j(x')) = f_{\pm}(x'),
  \end{equation}
  where $\sigma_j$ denotes the reflections in the coordinate
  hyperplanes of $\R^{n-1}$. The first assumption implies that
  there exists a function $f(x')$ so that the top of the domain is
  defined by $x_n = f(x')$ and the bottom is defined by $x_n = -
  f(x').$ The further symmetry assumptions then say that $f$ is an even function in every variable $x^i$, $1\leq i \leq n-1$, i.e.

\begin{equation} \label{f} f(x^1, \dots, x^{n-1}) = F({(x^1)}^2,
\dots, {(x^{n-1})}^2), \qquad F\in C^\omega(\mathbb R ^{n-1}).
\end{equation}

The value of these explicit formulae is demonstrated by
applications such as Theorem \ref{ONESYM}. Theorem \ref{BGAMMAJ}
may be viewed as an alternative to the use of Birkhoff normal
forms methods for calculating wave trace coefficients  as in
\cite{G,G2,Z1,Z2,ISZ,SZ}. Further discussion and comparison of
methods is given at the end of the introduction.

Our proof of Theorem \ref{BGAMMAJ} is  rather different from that
in \cite{Z3,Z4}. It is based on the construction and analysis of a
microlocal monodromy operator associated to $\gamma$,  inspired by
the works of Sj\"ostrand-Zworski \cite{SZ} and Cardoso-Popov
\cite{CP} (see also \cite{ISZ}), but employing a layer potential
analysis more closely related to that in  \cite{Z3,HZ}.
  The trace
asymptotics are  eventually reduced to those of a boundary
integral operator in Proposition \ref{NFORM} and Corollary
\ref{SOCINT}, and then to the stationary phase asymptotics of a
certain oscillatory integral in Theorem \ref{SETUPA}.  Since the
method and results give a higher dimensional generalization of the
analogous results  of \cite{Z3,Z4}, there is some overlap in the
arguments from the two-dimensional case; we have tried to minimize
the overlap, but it is necessary to give complete details on the
formulae in $n$ dimensions since they differ in numerous ways from
the two-dimensional case.

\subsection{Determining Taylor coefficients from wave invariants}
To prove Theorem \ref{ONESYM}, it is only necessary to determine
the Taylor coefficients of the defining function $f = f_+ = -f_-$
of $\Omega$ at the endpoints of a symmetric  bouncing ball orbit
$\gamma$ from the Balian-Bloch invariants given in Theorem
\ref{BGAMMAJ} for iterates of this orbit. This is done is   \S
\ref{SIMPLE}. The proof builds on the methods of \cite{Z3,H}.

In fact, our method could be extended to show that analytic
domains with fewer symmetries are spectrally determined as in
\cite{Z3,H}, but for the sake of brevity we do not prove that
here.

\subsection{Discussion and  Comparison of Methods}

We  obtain the formulae for the wave invariants by applying the
stationary phase method to the trace of a well-constructed
parametrix for the monodromy operator. A  secondary purpose of
this article is  to connect the very conceptual but somewhat
abstract monodromy method of \cite{ISZ,SZ} with the methods of
\cite{Z1,Z2,Z3}.
 The articles \cite{Z1,Z2}  implicitly used the monodromy approach
in the form given in \cite{BBa,L,LT}.   In this article, we
construct the monodromy operator explicitly in terms of layer
potentials,  using in part the methods of \cite{CP} and in part
those of \cite{Z3,HZ}.   The monodromy approach connects nicely
with the `Balian-Bloch' approach of \cite{Z3} and simplifies
remainder estimates for
 the Balian-Bloch (i.e. Neumann) expansion of the resolvent.

In calculating the trace asymptotics, we do not put the monodromy
operator into normal form, but rather apply a direct stationary
phase analysis to the parametrix.
 Terms of the stationary phase expansion correspond to
Feynman diagrams and the main idea (as in \cite{Z3}) is to isolate
the diagrams which are necessary and sufficient to determine the
Taylor coefficients of the boundary defining function at the
endpoints of $\gamma$  from the wave trace invariants of iterates
of $\gamma$. In this $(\Z/2\,\Z)^n$ symmetric case, there is a
unique such diagram and that is why the symmetries simplify the
problem. It is an interesting but difficult problem to `invert'
the spectrum when some or all of the symmetries are absent.

An  alternative to the approach of this article is to use  quantum
Birkhoff normal forms around the bouncing ball orbit as in
\cite{G,G2,Z1,Z2,ISZ,SZ}. It suffices to prove the abstract result
that the quantum normal form of the Laplacian or wave group at the
invariant bouncing ball orbit, hence that the classical Birkhoff
normal form of the Poincar\'e map, is a spectral invariant. At
that point, one could generalize the result of Y.Colin de
Verdi\`ere \cite{CV} that the classical normal form determines the
Taylor coefficients of $f$ at the endpoints of the bouncing ball
orbit when $f$ has the $(\Z/2 \Z)^n$ symmetries. We plan to carry
out the details in a follow-up to this article. The methods of
this article go further, since Theorem \ref{BGAMMAJ} determines
much more than the classical Birkhoff normal form.

It would be interesting to obtain a more direct connection between
the  the `normal forms' approach and the `parametrix approach'.
 In general terms, normal forms
for Hamiltonians and for canonical transformations   belong to the canonical Hamiltonian formulation of
quantum mechanics,  while  parametrix constructions, stationary phase methods  and Feynman diagrams
belong to the Lagrangian or path integral approach.   Normal forms are of course canonical, while
parametrices are not: there are many possible parametrices (finite dimensional approximations to path integrals),
and in the inverse problem it is essential to construct computable ones. The two approaches are   dual, and although they
contain the same information, it is formatted in different ways. In particular, the two approaches highlight  different
features of the geometry and dynamics.

At the present time, explicit calculations and spectral inversion
for boundary problems have only been carried out in the Lagrangian
approach, despite the existence of a  quantum normal form along
bouncing ball orbits \cite{Z2}. In the simpler setting of
semi-classical Schr\"odinger operators at equilibrium points, one
may compare the normal forms approach of \cite{GU,CVG} to the
Lagrangian approach of \cite{H}. In this inverse problem, one has
a one-parameter family of isospectral operators depending on a
Planck's constant $h$, whereas in the boundary problem one has
only one operator and spectrum to work with. The Lagrangian
calculations in \cite{H} reproduced the inverse results of
\cite{GU,CVG}, and gave  stronger ones where some of the
symmetries were removed. It directly gives formulae for wave
invariants, which are linear combinations of normal
 form invariants.

 It is interesting  to observe that formula in Theorem
\ref{BGAMMAJ} is very similar to the formula in \cite{H} (Theorem
2.1)  for the wave invariants at an equilibrium point for a
Schr\"odinger operator on $\R^n$ with a unique equilibrium point
at $(x,\xi)=(0,0)$. This perhaps indicates a similarity between
the quantum normal form of the Schr\"odinger operator at the
equilibrium point and that of the Laplacian at a bouncing ball
orbit. The formula is also similar to a trace asymptotics formula
of  T. Christiansen for an inverse problem for wave-guides
\cite{Chr}, but that is less surprising.

The methods of this paper have further
 applications. In a work in progress \cite{HeZ}, we use the wave
 invariants to prove a certain spectral rigidity result for
 analytic deformations of an ellipse.

 Finally, we would like to thank the referees for their
 suggestions on improving the exposition.

\section{\label{BACKGROUND} Background}

In this section, we go over the basic set-up of the problem. It is
very similar to that of \cite{Z3} but requires some higher
dimensional generalizations. We  use the same notation as in
\cite{Z3} and refer there for many details.

\subsection{\label{BILL} Billiard map}

The billiard map $\beta$ is defined on
$$B^*\partial \Omega=\{(y,\eta); \; y\in
\partial \Omega, \; \eta\in T^*\partial\Omega, \;|\eta| \leq 1\}$$ as
follows: given $(y, \eta) \in B^*\partial \Omega$, with $|\eta|
\leq 1$, let  $(y, \zeta) \in S^* \Omega$ be the unique
inward-pointing unit covector at $y$ which projects to  $(y,
\eta)$ under the map $T^*_{\pa \Omega} \Omegabar \to T^* \partial
\Omega$. Then  follow the geodesic (straight line) determined by
$(y, \zeta)$ to the first place it intersects the boundary again;
let $y' \in \partial \Omega$ denote this first intersection. (If
$|\eta| = 1$, then we let $y' = y$.) Denoting the inward unit
normal vector at $y'$ by $\nu_{y'}$, we let $\zeta' = \zeta + 2
(\zeta \cdot \nu_{y'}) \nu_{y'}$ be the direction of the geodesic
after elastic reflection at $y'$, and let $\eta'$ be the
projection of $\zeta'$ to $B^*_{y'}\partial \Omega$. Then we define
$$
\beta(y, \eta) = (y', \eta').
$$
The billiard map is a symplectic, hence measure preserving, map
with respect to the standard symplectic form on $T^*\partial
\Omega$. We denote its  graph of $\beta$ by
\begin{equation} \Cbill: =
\operatorname{graph} \beta \equiv \{ (\beta(z), z) \mid
z \in B^*
\partial \Omega \}. \label{billiardCR}\end{equation}
In the case of a convex domain,
\begin{equation} \label{DEFGAMMA} \Cbill = \Gamma_{d}: =  \{\big(y, -\nabla_{y} d(y, y'), y',  \nabla_{y'} d(y,y')
\big)\}, \end{equation} i.e. the Euclidean distance function $d(y,
y')$ is a generating function for $\beta$. For non-convex domains,
this graph is larger due to `ghost' billiard trajectories which
exit and re-enter $\Omega$ but satisfy the reflection law of equal
angles at each intersection point. Such ghost orbits are the price
one pays for using a parametrix and distance function   $d(y,y')$
defined on the ambient space $\R^n$.

\subsection{Length functional} We define the length functional on $(\partial
\Omega)^m$ (the Cartesian product), by

\begin{equation}\label{LENGTH}    L(y_1, \dots, y_m) =  |y_1 - y_2| + \dots + |y_{m-1} -y_m|+|y_m-y_1| .\end{equation}
Then  $L$ is a smooth away
 from the  `large diagonals' $\Delta_{p, p + 1}:= \{y_p = y_{p +
 1}\}$.
  The condition that $d L
= 0$ is the classical condition that each $2$-link defined by the
triplet $(y_{ p- 1}, y_p, y_{p + 1})$ makes equal angles with the
normal at $y_p$. Hence a smooth critical point corresponds to a
 closed $m$-link billiard trajectory. See for instance \S 2.1 of \cite{PS}.

\subsubsection{Length functional in graph coordinates near the iterates of a bouncing ball orbit}

A bouncing ball orbit $\gamma$ is a 2-link periodic trajectory of
the billiard flow, i.e. a reversible periodic billiard trajectory
that bounces back and forth along a line segment orthogonal to the
boundary at both endpoints. As in the Introduction we orient
$\Omega$ so that $\overline{AB}$ lies along the vertical $x_n$
axis, with $A = (0,\frac{L}{2}), B = (0, -\frac{L}{2}).$  We
parameterize the boundary locally as two graphs $x_n = f_{\pm}
(x')$ over
 the $x'$-hyperplane.
Thus, in  a small tube $T_{\epsilon}(\gamma)$ around
$\overline{AB}$, the boundary consists of two components, which
are  graphs of the form $ y = f_+(x')$ near $A$ and $y = f_-(x') $
near $B$.

We then define the  length functionals in Cartesian coordinates
for the two possible orientations of the  $r$th iterate of a
bouncing ball orbit by
\begin{equation} \label{LPM} {\mathcal L}_{\pm } (x_1', \dots, x_{2r}')
 = \sum_{p = 1}^{2r} \sqrt{(x_{p + 1}' - x_p')^2 +
 (f_{w_{\pm}(p + 1) }(x_{p +1}') - f_{w_{\pm}(p)}(x_p'))^2}. \end{equation}
 Here, $w_{\pm}: \Z_{2r} \to \{\pm\}$, where $w_{+}(p)$  (resp. $w_-(p)$) alternates sign starting with
 $w_+(1) = +$ (resp. $w_-(1) = -$).
Obviously the point $(x_1', \dots, x_{2r}')=(0,\dots,0)$,
corresponding to the $r$-the iteration of the bouncing ball orbit,
is a critical point of $\lcal_{\pm}$.

\subsection{\label{RWG} Resolvent and Wave group}

By the Dirichlet Laplacian $\Delta_{\Omega}^D$ we mean the
Laplacian $\Delta = \sum_{j = 1}^n \frac{\partial^2}{\partial
x_i^2}$ with domain $\{u \in H^1_0(\Omega): \Delta u \in L^2\}$;
thus, in our notation, $\Delta_{\Omega}$ is a negative operator.
The resolvent $R_{\Omega}^D(k)$  of the Laplacian
$\Delta_{\Omega}^D$ on $\Omega$ with Dirichlet boundary conditions
is the family of operators on $L^2(\Omega)$ defined for $k \in \C,
\Im k > 0$  by
$$R_{\Omega}^D(k) = - (\Delta_{\Omega}^D + k^2)^{-1}, \;\;\;\;\; \Im k > 0.$$ The resolvent kernel, which
we refer to as the   {\it Dirichlet Green's function}
$G_{\Omega}^D(k, x, y)$ of $\Omega \subset \R^n$, is by definition
the solution of the boundary problem:
\begin{equation}\label{GREEN}  \left\{ \begin{array}{l} (\Delta_x + k^2) G_{\Omega}^D(k, x, y) = - \delta(x - y),\;\;\;
(x, y \in \Omega) \\ \\
G_{\Omega}^D(k, x, y) = 0, \;\;\; x \in \partial \Omega.
\end{array} \right. \end{equation}
The discussion is similar for Neumann boundary conditions except
that its domain is $\{u \in H^2(\Omega): \partial_{\nu} u
|_{\partial \Omega} = 0\}$, and  $\partial_{\nu} G_{\Omega}^N(k,
x, y) = 0$ for $x \in
\partial \Omega$, where $\partial_{\nu}$ is the interior unit
normal. As in the introduction, we superscript functions of the
Laplacian with  $B$ to denote either boundary condition.

The resolvent with either boundary condition may be expressed  in
terms of the even wave operator
$E_\Omega^B(t)=\cos(t\sqrt{-\Delta_\Omega^B})$ as
\begin{equation} \label{LAPLACE} R_{\Omega}^B (k ) =  \frac{i}{k }
\int_0^{\infty} e^{i k t} E_\Omega^B(t) dt,\;\;\;\;\; (\Im k  > 0
)
\end{equation}
In this paper we will consider the resolvent $R_{\Omega}^B (k )$
along the logarithmic ray $k = \lambda + i \tau \log \lambda$,
where $\lambda > 1$ and $\tau \in \R^+$.

Given  $\hat{\rho}  \in C_0^{\infty}(\R^+)$, we define the
smoothed resolvent $R_{\Omega,\rho}^B(k)$ by

\begin{equation}\label{RR} R_{\Omega,\rho}^B(k):=\rho * (\mu R_{\Omega}^B (\mu))=
\int_{\mathbb R} \rho(k - \mu)(\mu R_{\Omega}^B(\mu)) \, d\mu.
\end{equation} We note that $\rho(k - \mu)$ is well-defined since
$\rho$ is an entire function. Let us discuss in what sense the
integral in (\ref{RR}) is defined. We notice that since $\mu \in
\mathbb R$ , we have defined the resolvent $R_{\Omega}^B (\mu)$ by
$R_{\Omega}^B (\mu+i0^+)$. Hence we can write
$$\begin{array}{lll}   R_{\Omega, \rho}^{B} (k) & =  & \lim_{\epsilon \to 0^+}\int_{\mathbb R} \rho(k - \mu)\mu R_{\Omega}^B(\mu+i\epsilon) \,
d\mu
\\& & \\
&=& -\lim_{\epsilon \to 0^+}\int_{k-\mathbb R} \rho(\mu')(k-\mu')
R_{\Omega}^B(k+i\epsilon-\mu')\ d\mu'\\& &\\ &=& \int_{\mathbb R}
\rho(\mu)(k-\mu) R_{\Omega}^B(k-\mu)\ d\mu,
\end{array}$$
where the last equality is obtained by taking an appropriate
contour. Thus we can also take this last integral as our
definition of the smoothed resolvent.

Now by (\ref{LAPLACE}) we can rewrite $R_{\Omega, \rho}^{B} (k)$
in terms of the wave operator as:
\begin{equation}\label{RTAU} \begin{array}{lll}   R_{\Omega, \rho}^{B} (k) & =  &
i \int_0^{\infty}\int_{\R} \rho(\mu)  e^{i (k-\mu) t} E_{\Omega}^B (t) dt d\mu \\& & \\
&=& i \int_0^{\infty} \hat{\rho} (t)   e^{i k t} E_{\Omega}^B (t) dt\\
&& \\
& = & \frac{i}{2} \big ( \rho(k +  \sqrt{-\Delta_{\Omega}^B}) + \rho(k -
\sqrt{-\Delta_{\Omega}^B}) \big ).
\end{array} \end{equation}

The Poisson formula for manifolds with boundary \cite{AM,GM,PS}
gives the existence of a singularity expansion for the trace of
$E_\Omega^B(t)$ near a transversal reflecting ray. If we
substitute this singularity expansion into the trace of
$(\ref{RTAU})$ we obtain an asymptotic expansion in inverse powers
of $k$ for the smoothed resolvent trace:

\begin{theo} \cite{AM,GM,PS} \label{BBL} Assume that $\gamma$ is a non-degenerate periodic reflecting ray,
 and let $\hat{\rho } \in C_0^{\infty}(L_{\gamma} - \epsilon, L_{\gamma} + \epsilon)$,
 such that $\hat{\rho} \equiv 1$ on $(L_{\gamma} - \epsilon/2, L_{\gamma} + \epsilon/2)$ and
with no other lengths in its support.  Then for
$k=\lambda+i\tau\log \lambda$ with $\tau \in \mathbb R^+$, $Tr
R_{\Omega, \rho}^{D} (k)$ admits a complete asymptotic expansion
of the form (\ref{PR}) where \begin{itemize}

\item ${\mathcal D}_{D, \gamma}(k)$ is the {\it symplectic
pre-factor} $${\mathcal D}_{D, \gamma}(k)  = C_0 \;
\epsilon_D(\gamma) \frac{e^{ i k  L_{\gamma}} e^{i \frac{\pi}{4}
m_{\gamma}}}{\sqrt{|\det (I - P_{\gamma})|}};$$

\item $P_{\gamma}$ is the Poincar\'e map associated to $\gamma$;

\item $\epsilon_D(\gamma)$ is the signed  number of intersections
of $\gamma$ with  $\partial \Omega$;

\item  $m_{\gamma}$ is the Maslov  index of $\gamma$;

\item $C_0$ is a universal constant (e.g. factors of $2 \pi$).
\end{itemize}
\end{theo}
\begin{defn}\label{BB} The coefficients $B_{\gamma, j}$ are called the wave
trace invariants (or Balian-Bloch  invariants) associated to the
periodic orbit $\gamma$.

\end{defn}

As emphasized above, the discussion in the case of the smoothed
Neumann resolvent $R_{\Omega, \rho}^{N} (k)$ is essentially the
same.

\section{\label{LP}  Resolvent and the layer potentials}

The method of layer potentials (\cite{T} II, \S 7. 11) solves
(\ref{GREEN}) in terms of the `layer potentials' $G_0(k, x, y),
\partial_{\nu_y} G_0(k, x, y) \in {\mathcal D}'(\Omega
\times \partial \Omega)$,  where $\nu$ is the interior unit normal
to $\Omega$,  and $\partial_{\nu} = \nu \cdot \nabla,$ and where
$G_0(k, x, y)$ is the `free' Green's function of $\R^n$, i.e. of
the kernel of the free resolvent $- (\Delta_0 + k^2)^{-1}$ of the
Laplacian $\Delta_0$ on $\R^n$. A key point, first recognized in
\cite{BB1,BB2} and  put on a rigorous mathematical basis in
\cite{Z3,HZ}, is that the layer potentials are semi-classical
(i.e. non-homogeneous) Lagrangian distributions in the $k$
parameter, with additional homogeneous singularities on the
diagonal. In effect we wish to make use of the first property and
suppress the second. This will be explained in \S \ref{Structure}.

First let us briefly recall the method of layer potentials.
 The free outgoing Green's function in dimension $n$ is given by:
\begin{equation}
G_0(k, x,y) =  \frac{i}{4} k^{n-2} (2 \pi k  |x-y|)^{-(n-2)/2}
\Ha_{n/2-1}(k | x-y |). \label{Hankel}\end{equation}   In general,
the Hankel function of index $\nu$ has the integral representation
(\cite{T}, Chapter 3, \S 6)
\begin{equation} \label{HANKEL} \begin{array}{lll} \Ha_{\,\nu} (z) & =&  (\frac{2}{\pi z})^{1/2} \frac{e^{i(z - \pi \nu/2 -  \pi/4)}}{ \Gamma(\nu + 1/2)}
\int_0^{\infty} e^{-s} s^{-1/2} (1 - \frac{s}{2 i z})^{\nu -1/2}
ds.
\end{array} \end{equation}

The single and double \textit{layer potentials}, as operators from
the boundary $\partial \Omega$ to the interior $\Omega$, are then
respectively defined by
\begin{equation} \left\{ \begin{array}{l} \SL f(x) =
\int_{\partial \Omega}  G_0(k, x,y) f(y) dS(y), \qquad \quad  x\in\Omega, \\ \\
\DL f(x) = \int_{\partial \Omega} 2\,\partial_{\nu_y} G_0(k, x, y)
f(y) dS(y), \quad x\in\Omega.
\end{array} \right.
\end{equation}
Similarly, for a function $f$ on $\partial \Omega$, the
\textit{boundary layer potentials} $S(k)$ and $N(k)$, as operators
from the boundary $\partial \Omega$ to itself are denoted by:
\begin{equation} \label{BLayer}  \left\{\begin{array}{ll} S(k) f(x) = \int_{\partial \Omega} G_0(k, x, y) f(y) dS(y),
\qquad \; \; \, \, \quad x\in\ \partial \Omega, \\
\\
N(k) f(x) = \,\int_{\partial \Omega}2\, \partial_{\nu_y} G_0 (k,
x, y) f(y) dS(y), \; \, \quad x\in\ \partial \Omega.
\end{array}\right.
\end{equation}

Given any function $g$ on $\mathbb R^n \backslash \partial \Omega$
and any $x\in \partial \Omega$, we denote by $g_+(x)$ (resp.
$g_-(x)$) the limits of $g(w)$ as $w \to x\in
\partial \Omega$ from $w \in \Omega$ (resp. $w \in \mathbb R^n \backslash
\bar{\Omega}$). The layer potentials and the boundary layer
potentials introduced above are related by the following

\begin{equation} \label{JUMPS} \begin{array}{l} (\SL f)_+(x) = (\SL f)_-(x) = S(k)
f(x),\\
\\
(\DL f)_{\pm}(x) = (\pm \,I + N(k)) f(x),\qquad \text{(The jump
formula)}.
\end{array}
\end{equation}
Using the above notation we also have the following interesting
formula of Fredholm-Neumann
\begin{equation}\label{FN}R_{\Omega}^D(\k) -
R_0(\k) = -\DL (I + N(k))^{-1} \SL^t,
\end{equation} where $R_0(k)$ is the
free resolvent. This is true because the
kernels of both sides of the equation are solutions
to the Helmholtz equation and also the restrictions to the boundary of these kernels
are the same by the jump formula (\ref{JUMPS}). The formula follows by the uniqueness of the solutions of the Helmholtz equation.

\subsection{Structure of the operator $N(k)$
}\label{Structure}

We now state the precise sense in which $N(k)$ is a semi-classical
Fourier integral operator quantizing the billiard map of $\partial
\Omega$ when $k=\lambda+i\tau\log \lambda $ and $\lambda=\Re k \to
\infty$. It additionally has homogeneous singularities on the
diagonal. The discussion here is adapted from \cite{HZ}. Note that
here our semiclassical parameter is $\frac{1}{k}$ which is a
complex parameter.

We denote $S^*\partial \Omega=\{ (y, \eta) \in T^*\partial \Omega;
\; |\eta| = 1 \}$ and define the diagonal set
\begin{equation}
\csd = \{ (z, z); \; z\in S^*\partial \Omega \} \subset T^*\partial \Omega \times T^* \partial
\Omega. \label{csd}\end{equation}

\begin{prop} \cite{HZ} \label{HZ} Assume that $\Omega$ is a smooth  domain.
Let $U$ be any neighborhood of $\csd$. Then there is a
decomposition $N(k)$ as
$$N(k) = N_{0}(k)  + N_{1}(k)+N_{2}(k),$$
where  $N_0(k)$ is a pseudodifferential operator of order $-1$ ,
$N_1(k)$ is a semi-classical Fourier integral operator of order
zero associated with the canonical relation $\Gamma_d$ (cf.
(\ref{DEFGAMMA})), and $N_2(k)$ has operator wavefront set
contained in $U$.
\end{prop}
This proposition is valid whether or not $\Omega$ is convex, but
 in the convex case  $\Gamma_d = \Cbill$ is the graph of the billiard
map (cf \S \ref{BILL}).

From the integral formula (\ref{HANKEL}), the Hankel function
$\Ha_{n/2-1}(z)$ is conormal at $z=0$ and as $z \to \infty$,
\begin{equation} \label{AMP} a(z): = e^{-iz} \Ha_{n/2-1}(z) \sim
z^{-1/2} \sum_{j = 0}^{\infty} a_j z^{-j}.
\end{equation}
The kernel of the single-layer potential then has the form
$$S(k, x, y) =G_0(k, x,y)= C k^{n-2} (k |x-y|)^{-(n-2)/2}
\Ha_{n/2-1}(k |x-y|).
$$
For the double layer potential it follows from the identity
$$\frac{d}{dz}\Ha_{\nu}(z)=\frac{\nu}{z}\Ha_{\nu}(z)-\Ha_{\nu+1}(z),$$that
$$N(k, x, y) = 2\,\partial_{\nu_y}G_0(k, x, y)=C k^{n-1} (k |x-y|)^{-(n-2)/2}
\Ha_{n/2}(k |x-y|) \; \langle \frac{x - y}{|x - y|}, \nu_{y}
\rangle.
$$
We now introduce cutoff functions as in \cite{HZ}. When $k =
\lambda + i \tau \log \lambda,$ we put
$$1 = \phi_1(|x-y|) + \phi_0(|k|^{3/4}|x-y|) + \phi_2(|x-y|, |k|)$$
where $\phi_1(t)$ is supported in $t \geq t_0$ for some
sufficiently small $t_0
> 0$ (see \cite{HZ} for the choice of this $t_0$), and $\phi_0(t)$
is equal to $1$ for $t \leq 1$ and equal to $0$ for $t \geq 2$.
(The power $3/4$ in $\phi_0$ could be replaced by any other power
strictly between $1/2$ and $1$). We define the operator $N_1(k)$
to be the one with the kernel $\phi_1(|x-y|) N(k,x,y)$ which has
the form
\begin{equation} \label{a1}
N_1(k,x,y)=C k^{(n-1)} e^{ik |x-y|} \phi_1(|x-y|) a_1 (k |x-y|) \;
\langle \frac{x - y}{|x - y|}, \nu_{y} \rangle,
\end{equation}
where $a_1(z)=z^{-\frac{n-2}{2}}a(z)$ has an expansion in inverse
powers of $z$ as $z \to \infty$ (see (\ref{AMP})), with leading
term $z^{-(n-1)/2}$. The operator $N_1(k)$ is manifestly a
semiclassical FIO of order $0$ with the phase function
$d(x,y)=|x-y|$, and thus associated with the canonical relation
$\Gamma_d$ (cf. (\ref{DEFGAMMA})). The operators $N_0(k)$ and
$N_2(k)$ are constructed similarly from the cut offs $\phi_0$ and
$\phi_2$ respectively. See \cite{HZ} for the details.


\section{Monodromy operators and boundary integral operators}

It simplifies the resolvent trace calculation considerably to
reduce it to a boundary trace. In \cite{Z4,Z3} this was done by
taking the direct sum $ R_{\Omega,\rho}^D(\k) \oplus
R_{\Omega^c,\rho}^N(k) $ of the interior Dirichlet and exterior
Neumann resolvents, and verifying that  $$Tr\big(
R_{\Omega,\rho}^D(\k) \oplus R_{\Omega^c,\rho}^N(k)  -
R_{0,\rho}(\k)\big) = \rho *
\frac{d}{d\k} \log \det (I + N(\k)).$$ The construction for
the interior Neumann resolvent and exterior Dirichlet resolvent is
essentially the same and is omitted for brevity.  This relation
was stated in the physics literature, and we refer to \cite{Z4}
for references. In this article, we take a related but somewhat
different method to reduce the trace to the boundary using
monodromy operator ideas similar to those of Cardoso-Popov
\cite{CP} and of Sj\"ostrand-Zworski \cite{SZ}. A novel feature is
that we relate the monodromy operators to the boundary integral
operators $N_1(k)$ and $N_0(k)$.

Before going into the details, let us note some motivating ideas.
First, both the monodromy operator $M(k)$ and the boundary
integral operator $N_1(k)$ are quantizations of the billiard map
in the sense of being semi-classical Fourier integral operators
whose canonical relation is $\Cbill$. This suggests that they must
be closely related. However, $M(k)$ is microlocally constructed
while $N_1(k)$ is global. Further, $N_1(k)$ is just a piece of
$N(k)$, which  has the more complicated structure described in \S
\ref{Structure}, and in the boundary reduction $N(k)$ is the
primary object. On the other hand, there is a simple exact formula
for $N(k)$ while $M(k)$ is only known through microlocal
conjugation to normal form. Hence our purpose here is to construct
a monodromy operator resembling  $M(k)$ using $N(k)$ and the layer
potentials. In doing so, we follow the approach to monodromy
operators  of \cite{CP}. Since the resulting monodromy operator
does not seem to arise from  the abstract set-up  of Grushin
reductions, the proof that the trace reduces to the boundary is
not the same as in \cite{SZ}.

\subsection{Definition of the monodromy operator}
In this section we will use the basic terminologies in
semi-classical analysis such as semi-classical pseudodifferential
operators, Fourier integral operators and semi-classical wave
front sets. We refer to \cite{GS,EZ,SZ,Al} for all the definitions
and properties. We just mention that in the following sections
when we write $T(k)\sim S(k)$, for two operators $T(k)$ and
$S(k)$, we mean $T(k)-S(k)$ is a negligible (residual) operator in
the sense that its kernel is of order $O(k^{-\infty})$ in all
$C^s$ norms.
\\

Let $\gamma$ be an $m$-link periodic reflecting ray, with vertices
at $v_j$ in $\partial \Omega$, where there are $m + 1$ vertices
and $v_{m + 1} = v_1$. Let $\eta_j$ be the projection to
$B_{v_j}^*\partial\Omega$ of the direction of the ray $\gamma$
where it hits the boundary at $v_j$. We denote
$$\partial \gamma:=\{(v_1,\eta_1),\dots (v_m,\eta_m)\}$$
Let $\Gamma_j$ be microlocal
neighborhoods of $(v_j,\eta_j)$ in $B^*\partial \Omega=\{(y,\eta)|
y\in
\partial \Omega, \; \eta\in T^*\partial\Omega, |\eta| \leq1\}$.
We then define a microlocalization to $\Gamma_j$ of the double
layer potential operator by
$$H_j(k):= \,\DL \Psi_{\Gamma_j} (k): \; C(\partial
\Omega) \to C(\Omega), \;\;\;
$$ where $\Psi_{\Gamma_j}(k)$ is a
$k$-pseudodifferential operator microsupported in $\Gamma_j$. In
fact we will choose, $$\Psi_{\Gamma_j}(k) = (I + N_0(k))^{-1}
\chi_{\Gamma_j}(k), $$ where $\chi_{\Gamma_j}$ is a microlocal cut
off supported in $ \Gamma_j$. More precisely
$\chi_{\Gamma_j}(k)=Op_{\frac{1}{k}}(a_{j}(y,\eta))$, where
supp($a_j(y,\eta)) \subset \Gamma_j$ and $a_j(y,\eta)=1$ in a
neighborhood of $(v_j,\eta_j)$. Here $(I + N_0(k))^{-1}$ is a
parametrix for $I+N_0(k)$. Notice a parametrix exists because
$N_0(k)$ is a pseudodifferential of order $-1$.

Clearly, by the jump formula $(\ref{JUMPS})$, for any $u \in
C(\partial \Omega),$
$$\left\{ \begin{array}{l} (-\Delta - k^2) H_j(\k) u
=  0, \\ \\
H_j(k)(u)|_{\partial \Omega} = (I + N(\k)) \Psi_{\Gamma_j}(k) u.
\end{array} \right.
$$  Next we put for every $1 \leq i \leq m$,
\begin{equation} P_i(\k) = H_i(\k) - H_{i+1}(\k) r_{\partial \Omega}
H_i(\k) + \cdots + (-1)^{m-1} H_{i+(m-1)}(\k) r_{\partial \Omega}
H_{i+(m-2)}(k) \cdots r_{\partial \Omega} H_i(k), \end{equation}
where $r_{\partial \Omega}: C(\Omega) \rightarrow C(\partial
\Omega)$ is the operator of restriction to the boundary. Also
notice by our notation the indices $j+m$ and $j$ are identified.
Then we define

\begin{equation}
P(\k)=\frac{1}{m} \sum _{i=1}^m P_i(\k).
\end{equation}

\begin{prop} We have: $(-\Delta - \k^2) P (\k) u
= 0$. \end{prop}

\begin{proof}

In fact for every $i$ we have $(-\Delta - \k^2) P_i (\k) u = 0$.
We observe that $P_i(k) = \DL Q_i(k)$ with
\begin{equation} \label{PK}
\begin{array}{l} \!\!\!\!\!Q_i(\k) = \Psi_{\Gamma_i} -
\Psi_{\Gamma_{i+2}} \big[(I + N(\k)) \Psi_{\Gamma_{i+1}}\big] +
\cdots +(-1)^{m-1} \Psi_{\Gamma_{i+m-1}}\big[(I +
N(\k))\Psi_{\Gamma_{i+m-2}} \cdots (I + N(\k))
\Psi_{\Gamma_i}\big].
\end{array} \end{equation}  The statement follows since $\DL$ maps $C(\partial \Omega)$
into the solutions of the Helmholtz equation.

\end{proof}

We now make a couple of useful technical observations:
\begin{prop} \label{NON2} We have:
 $$ r_{\partial \Omega}\,H_j(k) \sim   (I + N_0(\k) + N_1(k))
\Psi_{\Gamma_j}.$$
\end{prop}

\begin{proof} By the jump formula we know that $r_{\partial \Omega}\,H_j(k)=(I+N(k))\Psi_{\Gamma_j}$. The missing term $N_2$ has its wave front set
contained in $U \times U$ where $U$ is a small neighborhood
$S^*\partial\Omega$. But then $ WF' (N_2) \circ
WF'(\Psi_{\Gamma_j}) = \emptyset. $ So this term of the
composition is negligible in $k$.
\end{proof}

\begin{prop} \label{NPsi} For all $1\leq j < m$, we have:
$$\Psi_{\Gamma_{j+1}}\big[(I + N(\k))\Psi_{\Gamma_{j}} \cdots (I +
N(\k)) \Psi_{\Gamma_1}\big] \sim
\Psi_{\Gamma_{j+1}}\big[N_1(\k)\Psi_{\Gamma_{j}} \cdots N_1(\k)
\Psi_{\Gamma_1}\big].$$
\end{prop}

\begin{proof}
We argue inductively. For $j=1$, using the proof of Proposition
\ref{NON2} we have
$$\Psi_{\Gamma_{2}}\big[(I + N(\k))\Psi_{\Gamma_{1}}\big] \sim \Psi_{\Gamma_{2}}N_1(k)\Psi_{\Gamma_{1}}+\Psi_{\Gamma_{2}}\big[(I + N_0(\k))\Psi_{\Gamma_{1}}\big].$$
But the last term in the expression above is negligible in $k$, as
the two semi-classical pseudodifferentials $\Psi_{\Gamma_{2}}$ and
$ \big[(I + N_0(\k))\Psi_{\Gamma_{1}}\big]$ are microsupported in
the two disjoint open sets $\Gamma_2$ and $\Gamma_1$ respectively.
Now assume the statement is true for $j-1$. Then we write

$$\Psi_{\Gamma_{j+1}}\big[(I + N(\k))\Psi_{\Gamma_{j}} \cdots (I +
N(\k)) \Psi_{\Gamma_1}\big] \sim
\Psi_{\Gamma_{j+1}}\big[N_1(\k)\Psi_{\Gamma_{j}} \cdots N_1(\k)
\Psi_{\Gamma_1}\big] $$

$$ \qquad \qquad \qquad \qquad \qquad
\qquad \qquad \qquad \qquad \qquad \qquad+
\Psi_{\Gamma_{j+1}}(I+N_0(\k))\Psi_{\Gamma_{j}}\big[N_1
\Psi_{\Gamma_{j-1}} \cdots N_1(\k) \Psi_{\Gamma_1}\big].$$
Similarly, $\Psi_{\Gamma_{j+1}}(I+N_0(\k))\Psi_{\Gamma_{j}}$ is
negligible in $k$ and the second term above is negligible.
\end{proof}

Next we have the following important proposition which implicitly
defines the monodromy operator:

\begin{prop}\label{FIRSTLAST}  We have
$$\begin{array}{l} r_{\partial \Omega} P(\k) \sim I +
M(\k),\qquad (\text{microlocally near}\;\partial\gamma)
\end{array}
$$
where \begin{equation}\label{Monodromy}
M(k)=\frac{1}{m}\sum_{i=1}^m(-1)^{m-1}N_1(\k)\Psi_{\Gamma_{i+m-1}}
\cdots N_1(\k) \Psi_{\Gamma_i}.
\end{equation}
\end{prop}

\begin{defn} The monodromy operator for the $m$-link periodic reflecting ray $\gamma$
is the operator $M(k)$ on $L^2(\partial \Omega)$  defined by
$(\ref{Monodromy})$.
\end{defn}

\begin{proof} We define
$$ M_i(\k)=(-1)^{m-1}N_1(\k)\Psi_{\Gamma_{i+m-1}}
\cdots N_1(\k) \Psi_{\Gamma_i},$$ and we show that for every $i$
\begin{equation}\label{Mi} r_{\partial \Omega} P_i(\k) \sim \chi_{\Gamma_i} +
M_i(\k),\qquad (\text{microlocally
near}\;\partial\gamma).\end{equation}
Since$$\frac{1}{m}\sum_{i=1}^m M_i(\k)=M(\k),\quad \text{and}
\quad \frac{1}{m}\sum_{i=1}^m \chi_{\Gamma_i}=I \qquad
(\text{microlocally near}\;\partial\gamma),$$ taking averages of
equations ($\ref{Mi}$) over all $i$ implies Proposition
\ref{FIRSTLAST}. So we only prove ($\ref{Mi}$) for $i=1$. Since
$P_1(k) = \DL Q_1(k)$, by the jump formula (\ref{JUMPS}) and by
(\ref{PK}), we have
$$r_{\partial \Omega} P_1(\k)=(I+N(k))\sum_{j=0}^{m-1}(-1)^{j} \Psi_{\Gamma_{j+1}}\big[(I + N(\k))\Psi_{\Gamma_{j}}
\cdots (I + N(\k)) \Psi_{\Gamma_1}\big],$$ where the $0$-th term
of the sum is defined to be $\Psi_{\Gamma_1}$. Now we apply
Proposition \ref{NPsi} to each term of the above sum and we get

$$r_{\partial \Omega} P_1(\k)\sim(I+N(k))\sum_{j=0}^{m-1}(-1)^{j} \Psi_{\Gamma_{j+1}}\big[N_1(\k)\Psi_{\Gamma_{j}}
\cdots N_1(\k) \Psi_{\Gamma_1}\big]. $$ Hence by substituting
$I+N=(I+N_0)+N_1$ and collecting the terms corresponding to
the same number of iterations of the billiard map $\beta$ (this
number is the same as the number of factors $N_1$ in each term) we
obtain
$$1_{\partial \Omega} P_1(\k)\sim \sum_{j=1}^{m-1}(-1)^{j-1}( I-\chi_{\Gamma_{j+1}})\big[N_1(\k)\Psi_{\Gamma_{j}}
\cdots N_1(\k) \Psi_{\Gamma_1}\big]+\big(\chi_{\Gamma_1}  +
M_1(\k)\big).$$ The $j$-th term of the sum above is a $\k$-FIO
corresponding to $\beta^j$. We show that each of these terms is
microlocally equivalent to $0$ near $\partial\gamma$. Hence only
$\chi_{\Gamma_1}$ corresponding to $\beta^0$ and $ M_1(\k)$
corresponding to $\beta^m$ survive and the proposition follows.

Let us discuss why
$$( I-\chi_{\Gamma_{j+1}})\big[N_1(\k)\Psi_{\Gamma_{j}}\cdots N_1(\k) \Psi_{\Gamma_1}\big]\sim 0. \qquad (\text{microlocally near} \; \partial\gamma)$$
This is true because
$$ WF'(N_1(\k)\Psi_{\Gamma_{j}}\cdots N_1(\k) \Psi_{\Gamma_1})
\subset \{\big(\beta^j(y,\eta), (y,\eta)\big);\;(y,\eta)\in
\Gamma_1 \}\subset \Gamma_{j+1} \times \Gamma_1,$$ and because
$I-\chi_{\Gamma_{j+1}}$ is micro-supported away from
$(v_j,\eta_j)$. Thus by passing to a smaller microlocal
neighborhood of $\partial\gamma$ in $B^*\partial \Omega$, we
obtain a negligible operator.
\end{proof}

\subsection{Microlocal parametrix for the interior Dirichlet
problem in terms of $M(k)$}

In this section we construct a microlocal parametrix for the
Dirichlet resolvent of $\Omega$ near $\gamma$ in terms of the
monodromy operator $M(\k)$. It is a microlocal version of the
global formula (\ref{FN}) of Fredholm-Neumann. The discussion in
the Neumann case is similar.

\begin{prop}\label{PLAMBDAa}  Microlocally in a neighborhood of $\gamma \times \gamma \subset T^*\Omega\times T^*\Omega$,  we have
 $$ R_{\Omega}^D(\k) - R_0(\k)
 \sim - P(\k) (I + M (\k) )^{-1} \SL^t.$$
\end{prop}

\begin{proof}

Let us look at the Schwartz kernels of the both hand sides in a
microlocal neighborhood of $\gamma \times \gamma$. We show that

$$G_{\Omega}^D(\k,x,y) - G_0(\k,x,y)
 \sim  -\big(P(\k) (I + M (\k) )^{-1}\big) (\SL^t(\k,x,
 y)), \qquad (\text{microlocally near}\; \gamma \times \gamma)$$ where the operator $P(\k) (I + M (\k)
 )^{-1}: C(\partial \Omega) \rightarrow C(\Omega)$ acts on the
 first component $x$ of the kernel $\SL^t(\k,x,
 y)$. To prove this, since microlocal solutions are unique, it is
 enough to show that for all $y$ the right side is a solution of the Dirichlet problem
$$\left\{ \begin{array}{l} (-\Delta_x - \k^2)\big(P(\k) (I + M (\k) )^{-1}\big) (\SL^t(\k, x,
 y))
= 0, \\ \\
r_{\partial\Omega}\big(P(\k) (I + M (\k) )^{-1}\big) (\SL^t(\k, x,
 y))\sim-(G_{\Omega}^D -
 G_0)(\k,x,y), \quad \text{microlocally for}\; x\; \text{near}\; \partial \gamma.
\end{array} \right.
$$
But this is clear because by Proposition \ref{FIRSTLAST},
microlocally for $x\in\partial \Omega$ near $\partial \gamma$, we
have for all $y\in \Omega$
$$r_{\partial\Omega}\big(P(\k) (I + M (\k)
)^{-1}\big) (\SL^t(\k, x,
 y))\sim (I + M (\k)
)(I + M (\k) )^{-1} \SL^t(\k, x,
 y)$$
 $$\qquad \qquad \qquad \qquad \qquad \qquad \qquad \sim -  (G_{\Omega}^D -
 G_0)(\k,x,y).$$
 We note that the operator $I+M(k)$ is invertible by Lemma \ref{normofM}.
\end{proof}

\section{Trace formula and Monodromy operators}
 Here we assume $L_\gamma$ is the only length in the support of
$\hat \rho \in C^{\infty}_0$ and $\hat \rho(t)=1$ near $L_\gamma$.

From Proposition \ref{PLAMBDAa}, we immediately have a reduction
of the trace to the boundary.

\begin{prop} \label{SLDL} We have:

$$ Tr_{\Omega} (R_{\Omega,\rho}^D(k) - R_{0,\rho}(k)) \sim - \; Tr_{\partial \Omega} (\rho *\,k\, \SL^t \DL Q(k) (I + M(\k))^{-1}).$$

\end{prop}

\begin{proof}
First of all we note that the regularized trace $ Tr_{\partial
\Omega} (R_{\Omega,\rho}^D(k) - R_{0,\rho}(k))$ can be
microlocalized to $\gamma$, i.e. if $\chi_{\gamma}$ is a
microlocal cutoff around $\gamma$, then
$$ Tr_{\partial
\Omega}  (R_{\Omega,\rho}^D(k) - R_{0,\rho}(k)) \sim Tr_{\partial
\Omega} (\chi_{\gamma}(R_{\Omega,\rho}^D(k) - R_{0,\rho}(k))
\chi_{\gamma}). $$ For a proof of this fact, see \cite{Z4} \S
$3.3$. Now by Proposition \ref{PLAMBDAa}, we have
$$Tr _{\Omega}(R_{\Omega,\rho}^D(k) - R_{0,\rho}(k)) \sim - Tr_{\Omega} (\rho *\, k P(k)
(I + M(\k))^{-1}\SL^t)$$ $$ \qquad \qquad \qquad \qquad \qquad
\quad \sim - Tr_{\partial \Omega} (\rho * \,k \SL^t P(k) (I +
M(\k))^{-1}).$$ The formula follows by substituting $  P(\k) =
\;\DL Q(\k)$.
\end{proof}

This formula is useful but somewhat unwieldy. As proved in
\cite{Z3}, $\SL^t \DL = D_0 + D_1$ where $D_0$ is a
$\k$-pseudodifferential operator and where $D_1$ quantizes
$\beta$. The same analysis could be used here. But it is simpler
to use the alternative in the next section.

\subsection{Interior plus exterior}

If we take the direct sum of the interior Dirichlet and exterior
Neumann  resolvents, then the trace formula simplifies in that we
can sum up the interior and exterior $\SL^t \circ \DL$ operators
to obtain $\frac{1}{2k}N'(k):=\frac{1}{2k}(\partial/\partial
k)N(k)$.

\begin{prop} \label{SUM} We have:
$$ Tr_{\mathbb R^n}(R_{\Omega,\rho}^D(\k) \oplus R_{\Omega^c,\rho}^N(k)  -
R_{0,\rho}(\k)) = - \int_{\R} \rho(\mu)  Tr\big( N'(k-\mu)
Q(k-\mu) (I + M(k-\mu))^{-1} \big) d \mu.
$$\end{prop}

\begin{proof}

We first derive an analogue of Proposition \ref{PLAMBDAa} for the
exterior Neumann problem. We then take the trace of the direct sum
of the interior Dirichlet and exterior Neumann resolvents.

We construct a parametrix for the exterior Neumann problem by a
modification of the method used for the interior Dirichlet
problem. The discussion of the interior was motivated by the
double layer representation for the interior Dirichlet Green's
function. For the exterior Neumann problem we use the single layer
representation of the exterior Neumann Green's function,
$R^N_{\Omega^c}(k) - R_0(k) = -\SL (I + N^{t}(k))^{-1} \DL^t$,
where the superscript $t$ denotes the transpose. This formula is
proved by expressing the left side as $\SL \psi$ for some $\psi$,
taking the normal derivative from the exterior and solving for
$\psi$. We then consider $$r_{\Omega^c} (R^N_{\Omega^c}(k) -
R_0(k) ) r_{\Omega^c},$$ where $r_X$ is the charateristic function
of $X$. We observe that this operator is symmetric, i.e. equals
its transpose. It follows that $$\big(r_{\Omega^c}
(R^N_{\Omega^c}(k) - R_0(k) ) r_{\Omega^c}\big)(x,y) =
-\big(r_{\Omega^c} \DL (I + N)^{-1} \SL^t
r_{\Omega^c}\big)(y,x).$$ Therefore, at least on the diagonal, we
can use the same parametrix formula in the exterior.

We now complete the proof of Proposition \ref{SUM} by  taking the
trace on $\R^n$ of the direct sum of the two operators,
\begin{equation} \label{NS} \left\{ \begin{array}{l} r_{\Omega} (R_{\Omega}^D(\k) - R_0(\k))
r_{\Omega}
 \sim - r_{\Omega} P(\k) (I + M (\k) )^{-1} \SL^t r_{\Omega} \qquad (\text{microlocally near}\; \gamma) \\ \\
r_{\Omega^c} (R^{N}_{\Omega^c} (\k) - R_0(\k)) r_{\Omega^c}
 \sim -r_{\Omega^c} P(\k) (I + M (\k) )^{-1} \SL^t
 r_{\Omega^c} \qquad  (\text{microlocally near}\; \gamma)
\end{array} \right.
\end{equation}
In taking the trace we may cycle $ \SL^t r_{\Omega} $ to the front
in the first trace and $\SL^t r_{\Omega^c}$ to the front in the
second trace. We then add them to get,
\begin{equation}\begin{array}{lll}  Tr_{\R^n}  \left(R_{\Omega}^D(\k) \oplus R^{N}_{\Omega^c}
(\k)- R_0(k) \right) \sim -Tr_{\partial \Omega} \SL^t \DL Q(k) (I
+ M (\k) )^{-1}, \end{array} \end{equation} where $\SL^t \circ
\DL$ has the kernel
\begin{equation} \label{KERNEL} \int_{\R^n} G_0(k , x, w) \partial_{\nu_y} G_0(k, w, y) dw =
\frac{1}{2k}  N'(k, x, y).  \end{equation} For a proof of this
simple fact see equation (19) of \cite{Z5}. This indeed is why the
interior Dirichlet and exterior Neumann problems were combined and
explains the sense in which they are complementary. We then
convolve with $\rho.$
\end{proof}
\begin{rem} We notice that here the exterior trace $Tr \big(r_{\Omega^c} (R^{N}_{\Omega^c,\rho} (\k) - R_{0,\rho}(\k))
r_{\Omega^c}\big)$ is negligible in $k$ and it is added only to
simplify the expression in Prop \ref{SLDL} to the more convenient
expression in Prop \ref{SUM}.

\end{rem}

We now use Proposition \ref{SUM} to obtain asymptotics of the
trace. The next step is to expand $(I +M)^{-1}$ in a finite
geometric (Neumann) series with remainder. We have
\begin{equation} \label{GS} (I \!+\! M)^{-1} = \sum_{n = 0}^{n_0} (-1)^n \; M^n + (-1)^{n_0 + 1} \; M^{n_0 + 1}  (I \!+\! M)^{-1}.\end{equation}
The following proposition shows that, in
calculating a given order of Balian-Bloch invariant $B_{\gamma,
j}$, we may neglect a sufficiently high remainder of the expansion
(\ref{GS}).

\begin{prop} \label{ASY}  Assume that $k=\lambda+i\tau \log \lambda$. For each  order $|k|^{-J}$ in the trace expansion
 there exists $n_0(J)$ such that

$$
\begin{array}{ll} (i) \; \;Tr \int_{\R} \rho(\mu)  M(k-\mu)^{n_0(J) + 1}
(I \!+\! M(k-\mu))^{-1} N'(k-\mu) Q(k-\mu)  d \mu = O(|k|^{-J -
1}),\\ \\

(ii)\; \; Tr_{\Omega} R_{\Omega,\rho}^D(k)= \sum_{n = 0}^{n_0(J)}
(-1)^n Tr  \int_{\R} \rho(\mu)\; M(k-\mu)^n N'(k-\mu) Q(k-\mu) d
\mu+ O(|k|^{-J - 1}).
\end{array}$$
\end{prop}

\begin{proof} Part $(ii)$ is easily proved by combining part $(i)$, Proposition \ref{SUM} and (\ref{GS}).
It remains to estimate the remainder and show pat $(i)$. For this, we
need to establish an $L^2$- norm estimate for the operator norm of
$M(k-\mu)=M(\lambda-\mu + i \tau \log \lambda)$ for sufficiently
large $\tau$.

The proof is implied by the following norm estimate, which is analogous to Lemma 6.2 of
\cite{SZ}. Let $t_0$ be the constant in \S \ref{Structure}. We
note that the monodromy operator depends on a choice of $t_0$
although it is not indicated in the notation.

\begin{lem}\label{normofM} Let $k=\lambda+i\tau\log\lambda$. For every $a>0$ there exists $\tau>0$ and constants $b, C > 0$
such that,
$$||M(k-\mu)||_{L^2} \leq C |k|^{-a}<\mu>^{b}. $$ \end{lem}

To prove the Lemma,  we observe that by  Proposition \ref{FIRSTLAST},
 \begin{equation}\label{Monodromya}
||M(k-\mu)|| \leq C(|k| <\mu>)^b ||N_1(k-\mu)||^m,
\end{equation}
For some integers $C$ and $b$. We will not relabel these constants
in the course of our estimates.

We recall (see (\ref{a1})) that $N_1(k-\mu)$ has Schwartz kernel
\begin{equation} \label{a}
C(k-\mu)^{n-1} e^{i (k-\mu) |x-y|} \phi_1(|x-y|) a_1 ((\k-\mu)
|x-y|)\langle \frac{x - y}{|x - y|}, \nu_{y} \rangle,
\end{equation}
 where
$\phi_1(t)$ is supported in $t \geq t_0$ for some $t_0 > 0$. We estimate its norm by the Schur estimate,
\begin{equation} \begin{array}{lll} ||N_1(k-\mu)|| & \leq &  C|k-\mu|^{n-1}  \sup_{x \in \partial \Omega}
\int_{\partial \Omega} \big|e^{i (\lambda + i \tau \log \lambda)
|x-y|} \phi_1(|x-y|) a_1 ((k-\mu) |x-y|)\big| dS(y)
\\ && \\
& \leq & C |k-\mu|^{n-1} e^{- \tau \log \lambda t_0}  \sup_{x \in
\partial \Omega}
\int_{\partial \Omega}  \big| \phi_1(|x-y|) a_1 ((k-\mu)|x-y|)\big| dS(y)\\ && \\
& \leq & C <\mu>^{2n} e^{- (\tau t_0 - \epsilon) \log \lambda},
\end{array}
\end{equation} where we estimate $|k-\mu|^{n-1} \sup_{x, y \in \partial
\Omega, |x-y|\geq t_0} |a_1 ((k-\mu)  |x-y|)| \leq C
|k|^{2n}<\mu>^{2n}$.

Since we can choose any small $\epsilon > 0$ and also any large
$\tau
>0$, it is clear that, for any $a > 0$, there exist $\epsilon$ and
$\tau$ such that $||M(k-\mu)|| \leq C |k|^{- a}<\mu>^b. $ This
proves the Lemma and hence the first part of Proposition
\ref{ASY}.

\end{proof}

\subsection{Trace for the iterations of a bouncing ball orbit}

We now analyze the trace in part (ii) of Prop \ref{ASY} when it is
specialized to the $r$th iterate $\gamma^r$ of a bouncing ball
orbit, which has $m=2r$ links. We observe that $Q(\k-\mu)$  is a
sum of terms quantizing $\beta^0, \beta^1, \dots, \beta^{2r}$. Let
us write $q_j(k-\mu)$ for the term quantizing $\beta^j$. We note
that based on this notation, we have $q_{2r}(k-\mu)=M(k-\mu)$
where $M(k-\mu)$ is the monodromy operator for $\gamma^r$. It
follows that
$$N'(k-\mu) \, Q(\k-\mu) $$
is a sum of terms quantizing $\beta^0, \dots, \beta^{2r+1}$. On
the other hand if we use the monodromy operator for $\gamma^r$,
then $(I + M(\k-\mu))^{-1}$ is a sum of terms quantizing $\beta^0,
\beta^{2r}, \beta^{4r}, \cdots$. Therefore only three terms of
$N'(k-\mu)  Q(k-\mu)  (I + M(k-\mu))^{-1}$ are associated to
$\gamma^{2r}$ and can contribute to the trace:
\begin{enumerate}

\item $ - N_0'(\k-\mu) \,M(\k-\mu)$;

\item $N_0'(\k-\mu)\, q_{2r}(\k-\mu) =  N_0'(\k-\mu)\, M(\k-\mu)$;

\item $ N_1'(\k-\mu)\, q_{2r - 1}(\k-\mu) $.
\end{enumerate}
Here, we use that $N'_0$ is associated to $\beta^0$ and $N_1'$ is
associated to $\beta$. We notice the terms (1) and (2) cancel and
hence only the term (3) contributes to the trace. Also we notice
that for the $r$-th iteration of a bouncing ball orbit we have
only two vertices and therefore for $i$ odd we have
$\Gamma_i=\Gamma_1$ and for $i$ even we have $\Gamma_i=\Gamma_2$.
Let us denote by
$$ \Gamma_+=\Gamma_1, \quad \text{and} \quad \Gamma_-=\Gamma_2, $$
the microlocal neighborhoods corresponding to the top and bottom
vertex respectively. Thus, by these notations we have

\begin{prop}\label{NFORM} Let $\hat{\rho} \in C_0^{\infty}(\R)$ be a  cut off
 satisfying supp $\hat{\rho} \cap Lsp(\Omega) = \{ r L_{\gamma}
 \}$. Then

\begin{equation} \label{NONEK}\begin{array}{l}
Tr (R_{\Omega,\rho}^D(\k)) \sim -Tr\; \int_{\mathbb R} \rho(\mu)
N_1'(k-\mu) q_{2r-1} (k-\mu)\, d\mu
 \\ \\
 \, \, \quad \qquad \qquad \sim \;\frac{1}{2}\;
\sum_{\pm} Tr\; \int_{\mathbb R}\rho (\mu)\,  N_1'(k-\mu) \\ \\
  \, \, \quad \qquad \qquad \times\, \overbrace{(N_1
(k-\mu) \Psi_{\Gamma_{\mp}})(N_1(k-\mu) \Psi_{\Gamma_{\pm}})
\cdots  (N_1 (k-\mu) \Psi_{\Gamma_{\mp}})(N_1(k-\mu)
\Psi_{\Gamma_{\pm}})}^{(2r-1)\; \text{times}} \, d\mu. \\
\end{array}
\end{equation}
\end{prop}

We now express this trace as an explicit oscillatory integral. We
consider both principal (we will define the principal terms in the
next section) and non-principal terms. All terms arise as
composition of $2 r$ Fourier integral operators quantizing
$\beta$, hence may be expressed as compositions of $2r$
oscillatory integrals. We recall that $\Psi_{\Gamma_j} = (I +
N_0)^{-1} \chi_{\Gamma_j}$. Next we expand
$$(I + N_0(k))^{-1}  = I + N_{-1}(k), $$
where $N_{-1}(k)$ is a $(-1)$st order pseudo-differential
operator. If we plug $(I+N_{-1}(k))\chi_{\Gamma_{\pm}}$ for
$\Psi_{\Gamma_{\pm}}$ into the expression (\ref{NONEK}), after
expanding we get
\begin{cor} \label{SOCINT} Let $k = \lambda + i \tau \log \lambda$. Up to $O(|k|^{-\infty})$, the trace $Tr (R_{\Omega,\rho}^D(\k))$
is a sum of $2r$  oscillatory integrals of the form
$$ \begin{array}{l} I_{2r, \rho}^{\sigma} (k) =  \int_{\R} \int_{\R} \int_{(\partial \Omega)^{2r}}
e^{i [\mu(t-{\mathcal L}(y_1, \dots, y_{2r}))+k \mathcal L (y_1, \dots, y_{2r})]} \\ \\
A_{2r}^{\sigma}(k-\mu, y_1, \dots, y_{2r}) \hat{\rho}(t)\, dt\,
d\mu\, dS(y_1) \cdots d S(y_r),
\end{array} $$
where the superscript $\sigma$; $0\leq\sigma\leq 2r-1$, denotes
the sum of the terms which contain $\sigma$ factors of $N_{-1}$,
and where
$${\mathcal L}_{} (y_1, \dots, y_{2r}
) = |y_1 - y_2| + \cdots + |y_{2r} - y_1|,
$$ and $A_{2r}^{\sigma} (k- \mu , y_1, \dots, y_{2r})
\in S^{-|\sigma|  }_{\delta}((\partial \Omega) ^{2r})$.
\end{cor}

\section{Principal terms }
The goal of this section is to identify the  {\it principal
terms}, which generate the highest derivative data, and to prove
that non-principal terms contribute only lower order derivative
data.

As in \cite{Z3}, we separate out a single oscillatory integral
(the principal term $I^{0}_{2r,\rho}$)  which generates all terms
of the wave trace (or Balian-Bloch) expansion which contain the
maximal number of derivatives of the boundary defining function
per power of $k$ (i.e. order of wave invariant).

\begin{defn}\label{PRINCIPAL}  Let $\gamma$ be a $2$-link periodic
orbit, and let $\gamma^r$ be its $r$th iterate. The {\em principal
term} is the term of (\ref{NONEK}) in which $\Psi_{\Gamma_{\pm}}$
is replaced by $\chi_{\Gamma_{\pm}}$. Thus, the principal term is
$$ I^{0}_{2r,\rho}=-\; \sum _{\pm} Tr\; \rho \;*\;  N_1' (\;\overbrace{N_1 \chi_{\Gamma_{\mp}}\, N_1
\chi_{\Gamma_{\pm}} \cdots\, N_1 \chi_{\Gamma_{\mp}}\,N_1
\chi_{\Gamma_\pm}}^{(2r-1)\; times}\;) d\mu.
$$
\end{defn}

This oscillatory integral corresponds to $I^{0}_{2r,\rho}$, i.e.
the one in \ref{SOCINT} corresponding to $\sigma=0$. By Corollary
\ref{SOCINT}, the oscillatory integral $I^{0}_{2r,\rho}$ has the
phase function ${\mathcal L}_{} (y_1, \dots, y_{2r} )= |y_1 - y_2|
+ \cdots + |y_{2r} - y_1|, $ where $y_p \in
\partial \Omega$. We may write each $y_p $ in graph
coordinates as $(x'_p, f_{\pm}(x_p'))$. We will use superscripts
for the $n-1$ components of $x'_p$, i.e. $x'_p = (x_p^1, \dots,
x_p^{n-1})$. Hence the integral is localized to $[(-\epsilon,
\epsilon)^{n-1}]^{2r}$. We notice that $I^{0}_{2r,\rho}$ is the
sum of $I^{0,+}_{2r,\rho}$ and $I^{0,-}_{2r,\rho}$, where they
correspond to the $+$ and $-$ term respectively. It is clear that
the phase function of $I^{0,\pm}_{2r,\rho}$ is given by

$$ {\mathcal L}_{\pm} (x_1', \dots, x_{2r}')
 = \sum_{p = 1}^{2r} \sqrt{(x_{p + 1}' - x_p')^2 +
 (f_{\omega_{\pm}(p + 1) }(x_{p +1}') - f_{\omega_{\pm}(p)}(x_p'))^2}.$$
 Here, $w_{\pm}: \Z_{2r} \to \{\pm\}$, where $w_{+}(p)$
(resp. $w_-(p)$) alternates sign starting with
 $w_+(1) = +$ (resp. $w_-(1) = -$).
 \\

Now we have the following Theorem \ref{SETUPA}. First we have a
definition:
\begin{defn} \label{MODULO} Let $\gamma$ be an $m$-link periodic
reflecting ray, and let $\hat{\rho} \in C_0^{\infty}(\R)$ be a cut
off
 satisfying supp $\hat{\rho} \cap Lsp(\Omega) = \{ r L_{\gamma}
 \}$ for some fixed $r \in \N$. Given an oscillatory integral
 $I(k)$, we write
 $$Tr  R_{\Omega, \rho}^{B} (k) \equiv I(k)\;\;
 \mbox{mod} \;\; {\mathcal O}( \sum_{j} k^{-j} (\mathcal J ^{2j} f))$$
 if
 $$Tr R_{\Omega, \rho}^{B} (k) - I(k)$$
 has a complete asymptotic expansion of the form (\ref{PR}), and
 if the coefficient of $k^{-j}$ depends on $\leq 2j$
 derivatives of the defining functions $f$ at the reflection points.
\end{defn}

The following Theorem is the higher dimensional generalization of
Theorem 4.2 of \cite{Z3}.

\begin{theo}\label{SETUPA} Let $k=\lambda+i\tau \log \lambda$.  Let $\gamma$ be a primitive
non-degenerate $2$-link periodic
 reflecting ray, whose reflection points are points of non-zero curvature of $\partial \Omega$,
 and let $\hat{\rho} \in C_0^{\infty}(\R)$ be a cut off
 satisfying supp $\hat{\rho} \cap Lsp(\Omega) = \{ r L_{\gamma}
 \}$ and equals one near $r L_{\gamma}$ for some fixed $r \in \N$.  Orient $\Omega$ so that $\gamma$ is
 the vertical segment $\{x' = 0\} \cap \Omega$, and so that
  $\partial \Omega$ is a union of two graphs over $[-\epsilon, \epsilon]^{n-1}$. Then
  \\

\begin{enumerate}

\item $Tr R_{\Omega, \rho}^{B} (k) \equiv I_{2r,\rho}^{0}\;\;
 \mbox{mod} \;\; {\mathcal O}( \sum_{j} k^{-j} (j^{2j}
 f_{\pm}))$
 \\

\item We also have the following integral representation for
$I_{2r,\rho}^{0}$ in the $x_p'$ coordinates
\begin{equation}\label{EXPRESSIONA} I_{2r,\rho}^{0} =\!\!\!\;\;\;
\sum_{\pm} \int_{\left([-\epsilon, \epsilon]^{n-1}\right)^{2
r}}\!\!\!\!\;\;\; e^{i (k + i \tau) {\mathcal L}_{\pm}(  x_1',
\dots, x_{ 2r}')} a_{2r}^{pr,\pm}(k, x_1',  x_2', \dots, x_{2r}')
d x_1' \cdots dx_{2 r}',\\ [6pt] \hspace{-5pt}\end{equation} where
the phase $
 {\mathcal L}_{\pm}
  (x_1', \dots, x_{2r}')$
is given in (\ref{LPM}), and where the amplitude is given by:
 $$a_{2r}^{pr,\pm}(k, x_1', \dots, x_{2r}') =  {\mathcal L}_{w_{\pm}}(x_1', \dots, x_{2r}') A^{pr,\pm}_{2r}(k, x_1', \dots, x_{2r}')
+ \frac{1}{i}
 \frac{\partial}{\partial k}  A^{pr,\pm}_{2r} (k,  x_1',  \dots, x_{2r}'). $$
Here
\begin{equation}\label{AMPLSINGL}
\begin{array}{lll}
 A^{pr,\pm}_{2r}(k, x_1', \dots, x_{2r}') & = &     \Pi_{p = 1}^{2 r }\;
 \Big\{{a_1\Big(k\,\sqrt{(x_{p }' \!- \!x_{p\!+\!1}')^2 \!+\!
(f_{w_{\pm}(p)}(x_{p}') \!-\! f_{w_{\pm}(p \!+ \!1)}(x_{p\!+\!1}'
))^2}}\,\Big) \\ && \\ & \times &
 \ \frac{<\,x_{p}' -
x_{p+1}'\, , \, \nabla_{x_p'} f_{w_{\pm}(p)} (x_{p}')> -  (
f_{w_{\pm}(p)} (x_{p}') - f_{w_{\pm}(p+1)}(x_{p+1}') )}{
\sqrt{(x_{p}' - x_{p+1}')^2 + (f_{w_{\pm}(p)} (x_{p}') -
f_{w_{\pm}(p+1)'}(x_{p+1}') )^2}} \Big\}
\end{array} \end{equation}
 where $a_1$ is the Hankel amplitude in
(\ref{a1}). Here $w_{+}(p)=(-1)^{p+1}$ and $w_{-}(p)=-w_{+}(p)$.
Also we have identified $x_{2r + 1}' = x_1'$.
\end{enumerate}

\end{theo}
\begin{proof} To prove the first part of Theorem it is enough to show that for a given $\sigma \geq 1$, the coefficient of $k^{-j}$
in the stationary phase expansion of $I_{2r, \rho}^{\sigma} (k)$,
has only Taylor coefficients of order at most $2j-\sigma+1$. This is shown in
\S 5.4 of \cite {Z3}. The second part of Theorem follows from the
proof of Proposition 3.10 of \cite{Z3}. It is basically just
eliminating the variables $t$ and $\mu$ in the integral in
Corollary (\ref{SOCINT}) using the stationary phase lemma.

\end{proof} Theorem \ref{SETUPA} is a crucial ingredient in the
proof of Theorem \ref{ONESYM}. It gives explicit formula for the
phase and amplitude of the principal oscillatory integrals that
determine the highest order jet of $\Omega$ in each wave
invariant. The notation $A^{pr,\pm}_{2r}, a^{pr,\pm}_{2r}$ refers
to the amplitude of the principal terms of the $2r$th integral;
these amplitudes contain terms of all orders in $k$ and principal
here does not refer to the principal symbol, i.e. the leading
order term in the semi-classical expansion.

\section{\label{SPSECT} Stationary phase calculations of $I_{2r, \rho}^{0}$ and the wave invariants}

It is easy to see that (see Proposition $4.4$ of \cite{Z3}) we
have $I_{2r, \rho}^{0, +}=I_{2r, \rho}^{0, -}$ and therefore $I_{2r, \rho}^{0}=2I_{2r, \rho}^{0, +}$ . Hence it suffices
to consider the $+$ term. The oscillatory integrals $I_{2r,
\rho}^{0, +}$ have the form (\ref{EXPRESSIONA}) with the phase
${\mathcal L}_{+}$ and the amplitude $a_{2r}^{pr,+}$.

 The only critical point occurs when $x'_p = 0$ for all $p$.
   We denote by
Hess\,$\lcal_{\pm}(0)$ the $2r(n -1) \times 2r(n - 1)$ matrix with
components
\begin{equation} \text{Hess} \, \lcal_{\pm}(0) = \left(  \frac{\partial^2  \lcal_{\pm}}{\partial
x_{p}^{i}
\partial x_{q}^{j}}\right), \;\;  i, j = 1, \dots, n-1; \; p,q = 1, \dots, 2r. \end{equation}
It is  the Hessian of ${\mathcal L}_{\pm}$ at its critical point
$(x_1', \dots, x_{2r}') = 0$ in Cartesian coordinates.

We denote by ${\mathcal H}_{+}$  the inverse Hessian operator in
the variables $( x_1', \dots, x_{2r}')$ at this critical point.
That is ${\mathcal H}_{+} = \langle Hess({\mathcal L}_{+})^{-1}(0)
D, D \rangle,$ where $D$ is short for $(\frac{\partial}{\partial
x_{1}^1}, \cdots \frac{\partial}{\partial x_{2r}^{n-1}}).$ More
precisely,
\begin{equation} {\mathcal H}_{+} = \sum_{p, q = 1}^{2r} \sum_{i, j =
1}^{n-1} h^{(i,j), (p,q)} (\frac{\partial^2}{\partial
x_{p}^{i}
\partial x_{q}^{j}}). \end{equation}

Before we apply the Stationary Phase Lemma, in two subsections we
state some properties of the inverse Hessian matrix of $\mathcal L
_{+}$, and also some properties of the phase function $\mathcal L
_{+}$ and principal amplitude $a_{2r}^{pr,+}$ which may be derived
directly from the formula in Theorem \ref{SETUPA}.

\subsection{{Properties of Hess $({\mathcal
L}_{+})^{-1}$}}

Let $\{\nu_{j,\pm}\}_{j = 1}^{n-1}$ denote the eigenvalues of the
second fundamental form of $\partial \Omega$ at the endpoints of
the bouncing ball orbit. Without loss of generality we can assume

\begin{equation}\label{nu_j}
\nu_{j,\pm}= D^2_{{x^{j}}} f_{\pm}(0), \qquad j=1,..., n-1.
\end{equation}

This is because by an orthogonal change of variable (i.e. an
isometry of the plane) we can make Hess $f_{\pm}$ a diagonal
matrix. Of course when all the symmetry assumptions are satisfied,
then Hess $f_{\pm}$ is automatically diagonal.


The following generalizes Proposition 2.2 of \cite{Z3}.

\begin{prop} \label{KT}
Put  $ a_{j,\pm} = -2(1\pm L \nu_{j,\pm} ),$ and let $A_{\pm} =
Diag(a_{j,\pm})$ be the $(n-1) \times (n-1)$  diagonal matrix with
the diagonal entries $a_{j,\pm}$  Then the Hessian $H_{2 r}$ of
${\mathcal L}_+$ at $x' = 0$ has the form:
$$ \label{HBB} H_{2r} = \frac{-1}{L} \left\{ \begin{array}{lllll} A_+  & I & 0 & \dots & I  \\ & & & & \\
  I &  A_-  & I & \dots & 0 \\ & & & & \\ 0 & I &   A_+ & I & 0  \\ & & & &
   \\ \dots & \dots & \dots & \dots & \dots \\ & & & & \\
I & 0 & 0 & \dots &  A_-\end{array} \right\},$$ where there are $2
r \times 2r$ blocks and each block is of size $(n-1 )\times
(n-1)$.
\end{prop}

\begin{proof} There are $2r$ sets of variables $x'_p$ and therefore there are $2r \times 2r$ blocks
 and the $(p,q)$-th block
is given by $D^2_{x_p' x_{q}'}\mathcal L_+(0)$. We have:
\begin{equation} \label{crit} \begin{array}{l} \nabla_{x_p'} {\mathcal L}_{\pm} =
 \frac{(x_p' - x_{p + 1}') + (f_{w_{\pm}(p)}(x_p')-f_{w_{\pm}(p +
 1)}
 (x_{p + 1}')) \nabla_{x_p'} f_{w_{\pm}(p)}(x_p')}{\sqrt{(x_p' - x_{p + 1}')^2 +
(f_{w_{\pm}(p)}(x_p') - f_{w_{\pm}(p + 1)}(x_{p + 1}'))^2}} -
\frac{(x_{p - 1}' - x_{p }') + (f_{w_{\pm} (p - 1)} (x_{p - 1}') -
f_{w_{\pm}(p)}(x_{p}')) \nabla _{x_p'}f_{w_{\pm}
(p)}(x_p')}{\sqrt{(x_p' - x_{p - 1}')^2 + (f_{w_{\pm}(p)}(x_p') -
f_{w_{\pm}(p - 1)}(x_{p - 1}'))^2}}.
\end{array} \end{equation} A simple
calculation (using $(\ref{crit})$ ) shows that all the blocks
$D^2_{x_p' x_{q}'}\mathcal L_+(0)$ are zero except the ones with
$p=q, p=q+1$ and $q=p+1$. From $(\ref{crit})$ we obtain
$$ \left\{ \begin{array}{l}  D^2_{x_p'x_p'} {\mathcal L}_+(0) =
 2\big( \frac{1}{L}\,I + w_+(p) \text{Hess}\,f_{w_+(p)}(0)\big)=\frac{-1}{L}\,A_{w_+(p)},\;\;\;\; \\ \\
D^2_{x_p' x_{p+1}'}  {\mathcal L}_+ (0) =
  \frac{-1}{L}\, \,I. \end{array}\right.$$
\end{proof}

 In the elliptic case, $\det H_{2r}$ is a polynomial in $\cos \al_j/2$\,
 (in $\cosh \al_j/2$ in the hyperbolic case)\, of degree
 $2r(n-1)$. Here, in the elliptic case, $\{e^{\pm i \al_1},...e^{\pm i \al_{n-1}}\}$ are the
eigenvalues of the Poincare map $ P_{\gamma}$.
\begin{prop}  \label{DetHESS}We have
$$\det H_{2r} = L^{2r(1-n)} \prod_{j=1}^{n-1}(2 - 2  \cos r \al_j). $$ \end{prop}

We will use Proposition $\ref{Det(I-P)}$ in the following
subsection to prove Proposition $\ref{DetHESS}$.

\subsubsection{Poincar\'e map and Hessian of the length functional}
The linear Poincar\'e map $P_{\gamma}$ of $\gamma$ is the
derivative at $\gamma(0)$ of the first return map to a transversal
to $\Phi^t$ at $\gamma(0).$ By a non-degenerate periodic
reflecting ray $\gamma$, we mean one whose linear Poincar\'e map
$P_{\gamma}$ has no eigenvalue equal to one. For the definitions
and background, we refer to \cite{PS, KT}.

There is an important relation between the spectrum of the
Poincar\'e map $P_{\gamma}$ of a periodic $m$-link reflecting ray
and the Hessian $H_m$ of the length functional at the
corresponding critical point of $L: (\partial \Omega)^m \to \R.$ For the
following, see \cite{KT} (Theorem 3).

\begin{prop} \label{Det(I-P)} Let $\gamma$ be a periodic $m$-link reflecting ray in plane domain $\Omega$.
Then we have: $$\det (I - P_{\gamma}) = \det (H_m) (\prod_{p=1}^m {H_{p,p+1}})^{-1},$$ where $H_{p,p+1}$ is the $(p,p+1)$-th entry of $H_m$.
\end{prop}

Proposition $\ref{Det(I-P)}$ is stated only for the plane domains.
One can probably prove it for higher dimensions, but the formulae
for the plane domains is enough for us to prove
Prop.$\ref{DetHESS}$
\\

\emph{Proof of Prop.$\ref{DetHESS}$.}\, Let us first assume $n=2$.
Let $\lambda_{r}, \lambda_{r}^{-1}$ be the eigenvalues of
$P_{\gamma^r}$, so that $\det (I - P_{\gamma^r}) = 2 - (\lambda_r
+ \lambda^{-1}_{r}).$ Since in our case $H_{p,p+1}=-1/L$, from Prop.$\ref{Det(I-P)}$ it follows that
\begin{equation} \label{HESSPOIN} \det (I - P_{\gamma^r}) = L^{2r} \det H_{2r};\;\;\; (\gamma \;\; 2-\mbox{link}.) \end{equation}
\\

This is because if the eigenvalues of $P_{\gamma}$ are $\{e^{\pm i \alpha}\}$
(say, in the elliptic case) then those of $P_{\gamma^r}$ are
$\{e^{\pm i r \alpha}\}$, hence the left side of (\ref{HESSPOIN})
equals $2- 2 \cos r \alpha.$ Now assume $n\geq 2$ and assume
$\{e^{\pm i \alpha_1},... e^{\pm i \alpha_{n-1}}\}$ are the
eigenvalues of $P_{\gamma}$. We just showed that if we define

$$ \label{HBBj} H_{j,2r} = \frac{-1}{L} \left\{ \begin{array}{lllll} a_{j,+}  & 1 & 0 & \dots & 1  \\ & & & & \\
  1&  a_{j,-}  & 1 & \dots & 0 \\ & & & & \\ 0 & 1 &  a_{j,+} & 1 & 0  \\ & & & &
   \\ \dots & \dots & \dots & \dots & \dots \\ & & & & \\
1 & 0 & 0 & \dots &  a_{j,-}\end{array} \right\}_{2r\times 2r},$$

then $\text{det}H_{j,2r}=L^{-2r}(2-2\cos(r\al_j))$. Now we notice
because all the blocks of the matrix $H_{2r}$ are diagonal
matrices, therefore they commute and we can write

$$ \text{det}H_{2r}=\text{det} \big(\text{Diag}(\text{det}H_{1,2r}, \dots, \text{det}H_{n-1,2r})\big)
=L^{2r(1-n)} \prod_{j=1}^{n-1}(2 - 2  \cos r \al_j).$$
We now consider the inverse Hessian ${\mathcal H}_+ =
H_{2r}^{-1}$, which will be important in the calculation of wave
invariants. We denote its matrix elements by $h^{ij,pq}_+$ which
corresponds to the $(i,j)$-th entry of the $(p,q)$-th block of the
matrix ${\mathcal H}_+$ . We also denote by ${\mathcal H}_-$ the
matrix in which the roles of $A_+, A_-$ are interchanged; it is
the Hessian of ${\mathcal L}_-$. We also notice since $H_{2r}$ is
a block matrix with each block a diagonal matrix so is its inverse
${\mathcal H}_+$. Hence the only non-zero entries of the inverse
Hessian ${\mathcal H}_+$  are of the form $h^{ii,pq}_+$.

\begin{prop} \label{HESSPAR}  The diagonal matrix elements $h^{ii,pp}_{+}$ are
constant when the parity of $p$ is fixed, and for every $1\leq
i\leq n-1$ we have:
$$\begin{array}{llll} p \;\; \mbox{odd} \;\; \implies & h^{ii,pp}_{\pm} = h^{ii,11}_{\pm},\;\;\; &
 p \;\; \mbox{even} \;\; \implies & h^{ii,pp}_{\pm} =
 h^{ii,22}_{\pm}
 \\ & & & \\
h^{ii,11}_+ = h^{ii,22}_-, & h^{ii,22}_+ = h^{ii,11}_- & & .
\end{array}$$
\end{prop}

\begin{proof}

It is enough to show this for $n=2$. This is because $H_{2r}$ is a
block matrix with commuting blocks and each block is diagonal. In
fact based on our definition in $(\ref{HBBj})$ we have

\begin{equation} \label{BIGHESSIAN} h^{ii,pq}=\big(H_{2r}\big)^{ii,pq}=\big(H_{i,2r}\big)^{pq}.\end{equation}
Hence it is enough to prove the proposition for the entries of the
inverse of the $2r\times 2r$ matrix $H_{i,2r}$. This reduces the
problem to $n-1=1$.

Now let us introduce the cyclic shift operator on $\R^{2r}$ given
by $P e_j = e_{j + 1}$, where $\{e_j\}$ is the standard basis, and
where $P e_{2r} = e_1.$ It is then easy to check that $ P
{\mathcal H}_+ P^{-1} = {\mathcal H}_-,$ hence that $ P {\mathcal
H}_+^{-1} P^{-1} = {\mathcal H}_-^{-1}.$ Since $P$ is unitary,
this says
$$h^{pq}_- = \langle {\mathcal H}_- e_p, e_q \rangle =  \langle  P {\mathcal H}_+^{-1} P^{-1} e_p, e_q \rangle =
 \langle   {\mathcal H}_+^{-1} P^{-1} e_p, P^{-1} e_q \rangle = h_+^{p -1, q - 1}.$$
 It follows that the matrix ${\mathcal H}_{\pm}$ is
invariant under even powers of the shift operator, which shifts
the indices $p\to p + 2 k$ ($k = 1, \dots, r$). Hence, diagonal
matrix elements of like parity are equal.

\end{proof}

\subsubsection{Inverse Hessian at $({\Z}/{{2\Z}})^n$-symmetric bouncing ball orbits}

We first observe that in the case of $({\Z}/{{2\Z}})^n$
symmetric domains, the $(2r)  (n-1) \times (2r) (n-1)$ Hessian of
Proposition \ref{KT} simplifies to:
\begin{equation} \label{HBB2} H_{2r} = \frac{-1}{L} \left\{ \begin{array}{lllll} A & I & 0 & \dots & I  \\ & & & & \\
  I & A & I & \dots & 0 \\ & & & & \\ 0 & I &  A & I & 0  \\ & & & & \\ 0 & 0 & I & A& I \dots
 \\ & & & & \\\dots & \dots & \dots & \dots & \dots \\ & & & & \\
I & 0 & 0 & \dots &  A\end{array} \right\} \end{equation} which is
a  $2 r\, \times\, 2r$ block matrix, each block of size
$(n-1)\times (n-1)$. Here $A = Diag(a_j)$ and $I$ is the rank
$n-1$ identity matrix. The diagonal entries $a_j$ are given by
\begin{equation} \label{FLOQUET}a_j =2\cos \alpha_j/2 = -2 (1 + L \nu_j) \;\;(\mbox{elliptic
case}),\;\;\;a_j=2\cosh \alpha_j/2 = -2 (1 + L \nu_j)
\;\;(\mbox{hyperbolic case}). \end{equation}

We can express the inverse Hessian matrix elements
$h_{2r}^{ij,pq}$ in terms of Chebychev polynomials $T_m,$ resp.
$U_m$,  of the first, resp. second, kind. The Chebychev
polynomials are defined by:
$$T_m(\cos \theta) = \cos m \theta,\;\;\;\;\; U_m(\cos \theta) = \frac{\sin (m + 1) \theta}{\sin \theta}.$$

\begin{prop}\cite{Z3}
Suppose that $H_{2r}$ is given by (\ref{HBB2}). Then the matrix elements of $H^{-1}_{2r}$
are given by
$$\begin{array}{ll} h^{ii,pq}_{2r} = \frac{1}{2[ 1 - T_{2r}(- a_i/2)]} [U_{2r - q + p -1}(-a_i/2) +
U_{q - p - 1}(-a_i/2)],\;\;\;\; (1 \leq p \leq q \leq 2r; \;\, 1\leq i\leq n-1) \\ \\
h^{ij,pq}_{2r}=0, \qquad i \neq j. \end{array}$$ \end{prop}

\begin{proof}
This formula was proved in \cite{Z3} for the case $n=2$. For the
general case $n\geq 2$, we just use the equation
(\ref{BIGHESSIAN}) and reduce it to $n=2$.
\end{proof}

We note that $h^{ij,pq} = h^{ij,qp}$ so this formula determines all of
the matrix elements.
It follows in the elliptic case that
\begin{equation}\label{HPQ} \begin{array}{l} h_{2r}^{ii,pq}
= \left\{ \begin{array}{ll}  \frac{(-1)^{p - q}}{2( 1 - \cos r \alpha_i)} \big( \frac{\sin (2r - q + p) \alpha_i/2}{\sin \alpha_i/2}  +  \frac{\sin ( q - p ) \alpha_i/2}{\sin \alpha_i/2}\big) & (1 \leq p \leq q \leq 2r) \\ & \\
\frac{(-1)^{p - q}}{2( 1 - \cos r \alpha_i)} \big( \frac{\sin (2r
- p + q) \alpha_i/2}{\sin \alpha_i/2}  +  \frac{\sin ( p - q )
\alpha_i/2}{\sin \alpha_i/2} \big) & (1 \leq q \leq p \leq 2r)
\end{array} \right. \end{array} \end{equation} There are obvious
analogues in the hyperbolic and mixed cases.

The case of interest to us is
\begin{equation}\label{HPQ11} h_{2r}^{ii,11}
= \frac{1}{2( 1 - \cos r \alpha_i)} \frac{\sin r
\alpha_i}{\sin \alpha_i/2} = \frac{1}{\sin \alpha_i/2} \cot
\frac{r \alpha_i}{2}.
\end{equation}

\subsection{{Properties of the phase function ${\mathcal
L}_{+}$ and the amplitude  $a^{pr,+}_{2r}$}}  Since ${\mathcal
L}_{+}$ and $a^{pr,+}_{2r}$ are functions of $2r(n-1)$ variables
$x_p^j$ where $1\leq p \leq 2r$ and $1\leq j \leq n-1$, to
simplify our notations we denote:
\\

Let $[\gamma]=(\gamma_p^{j})$, \, $1\leq p \leq 2r$, $1\leq j \leq
n-1$  \, be a $2r \times(n-1)$ matrix of indices. We let $m=
|[\gamma]|=\sum_{p,j} \gamma^j_p$. Then we define
$$D^m_{[\gamma]}= \frac{\pa^{m}}{(\pa x^1_1)^{\gamma^1_1}...(\pa
x_{2r}^{n-1})^{\gamma^{n-1}_{2r}}}.$$ We will use $\vec \gamma_p$
for the $p$-th row of $[\gamma]$, and sometimes if $\vec
\gamma_{p},\vec \gamma_{q},...$ are the only non-zero rows of
$[\gamma]$ we write $D^m_{\vec \gamma_{p}, \vec \gamma_{q},...}$
for $D^m_{[\gamma]}$ to emphasize that $\vec \gamma_{p},\vec
\gamma_{q},...$ are the only non-zero rows of $[\gamma]$. The
calculation of the highest derivative terms of the Balian-Bloch
wave invariants uses only the following properties of the phase
and principal amplitude which may be derived directly from the
formulae in Theorem \ref{SETUPA}.

The following Lemma is the higher dimensional generalization of
Lemma 4.5 of \cite{Z3}. It is proved in the same way, and the
proof is therefore omitted.

\begin{lem}\label{AMPPROPS}  The phase and principal amplitude of the
principal oscillatory integrals $I_{2r, \rho}^{0, \pm}$
have the following properties:

\begin{itemize}
\item[1.] In  its dependence on the boundary
defining functions $f_{\pm}$, the amplitude
$a^{pr,+}_{ 2r}$
 has the form $\alpha_r (k, x',  f_{\pm}, f_{\pm}')$.

\item[2.] As above, in its dependence on $x'$
 $$\qquad \; \;\; a^{pr,+}_{2r
 } (k, x_1', \dots, x_{2r}')   =  {\mathcal L}_{+}(x_1', \dots, x_{2r}') A^{pr,+}_{2r}(k, x_1', \dots, x_{2r}')
+ \frac{1}{i}
 \frac{\partial }{\partial k} A^{pr,+}_{2 r} (k,   x_1',  \dots, x_{2r}'),$$where $$A^{pr,+}_{2
r} (k, x_1', \dots, x_{2r}') \; = \; \Pi_{p = 1}^{2r} A_p (x_{p}',
x_{p + 1}'),\;\;\;\; 2 r + 1 \equiv 1.\;\;(\text{see} \; (\ref{AMPLSINGL}))$$

\item[3.] At the critical point, the principal amplitude has the
asymptotics
$$a^{pr,+}_{2r}(k, 0) \sim (2 r L) L^{-(n-1)r} \acal(r) + O(k^{-1}),$$
where $\acal(r)$
depends only on $r$ and not on $\Omega$.

\item[4.] $\frac{a^{pr,+}_{2r}(k, 0) e^{i (k + i \tau) {\mathcal
L_+}(0) + i \pi/4 sgn Hess {\mathcal L_+}(0)} }{\sqrt{\det Hess
{\mathcal L}_+}} \sim (2 r L)\;{\acal(r)}\;
{\mathcal D}_{D, \gamma^r}(k) (1 +  O(k^{-1})) \; (cf. \ref{PR})$.

\item[5.] $\nabla a^{pr,+}_{2r} (k, x_1', \dots, x_{2r}')|_{x' = 0} = 0$.

\item[6.] $D^{2j+1}_{\vec \gamma_p} {\mathcal L}_+ |_{x'=0} \equiv
2 w_+(p) D^{2j+1}_{\vec \gamma_p}f_{w_+(p)}(0)\;\;\; \mbox{mod}
\;\; R_{2r} ({\mathcal J}^{2j} f_+(0), {\mathcal
J}^{2j} f_-(0))$.

\item[7.] $D^{2j+2}_{\vec \gamma_p} {\mathcal L}_+ |_{x=0} \equiv 2
w_+(p) D^{2j+2}_{\vec \gamma_p}f_{w_+(p)}(0) \;\;\; \mbox{mod} \;\;
R_{2r} ({\mathcal J}^{2j} f_+(0), {\mathcal J}^{2j} f_-(0))$.

\item[8.] If $[\gamma]$ has more than one non-zero
row, say $\vec \gamma_p, \vec \gamma_q,...$, then
$$D^{2j+1}_{\vec \gamma_p, \vec \gamma_q, ...
}{\mathcal L}_+ (0) \equiv 0 \;\;\; \mbox{mod} \;\; R_{2r}
({\mathcal J}^{2j} f_+(0), {\mathcal J}^{2j} f_-(0)),$$
and
$$D^{2j+2}_{\vec \gamma_p, \vec \gamma_q, ...
}{\mathcal L}_+ (0) \equiv 0 \;\;\; \mbox{mod} \;\; R_{2r}
({\mathcal J}^{2j} f_+(0), {\mathcal J}^{2j} f_-(0)).$$
\end{itemize}

Above, $\equiv$ means equality modulo lower order
derivatives of $f$.

\end{lem}

\subsection{\label{SPFD} Stationary phase diagrammatics}

We briefly review the stationary phase expansion from the
diagrammatic point of view. For more details we refer to
\cite{A,E,Z3}.

The stationary phase expansion gives an asymptotic expansion for
an oscillatory integral
$$Z(k) = \int_{\R^n} a(x) e^{ik S(x)} dx$$
where $a \in C_0^{\infty}(\R^n)$ and where $S$ has a unique
non-degenerate critical point in supp$(a)$  at $x=0$. Let us write
$H$ for the Hessian of $S$ at $0$. The stationary phase expansion
is:
$$\displaystyle Z(k) = (\frac{2\pi}{k})^{n/2} \frac{e^{i \pi sgn
(H)/4}}{\sqrt{|det H|}} e^{i k S(0)} Z_A(k),$$ where
$$Z_A(k)=\sum_{j = 0}^{\infty} k^{-j} \sum_{
  (\Gamma, \ell)  \in {\mathcal G}_{V, I},\;
  {I-V=j}}
\frac{I_{\ell} (\Gamma)}{S(\Gamma)},$$ where ${\mathcal G}_{V, I}$
consists of  labelled graphs $(\Gamma, \ell)$ with $V$ closed
vertices of valency $\geq 3$ (each corresponding to the phase),
with one open vertex (corresponding to the amplitude), and with
$I$ edges. The function $\ell$ `labels' each end of each edge of
$\Gamma$ with an index $j \in \{1, \dots, n\}.$

Above, $S(\Gamma)$ denotes  the order of the automorphism group of
$\Gamma$, and  $I_{\ell} (\Gamma)$ denotes the `Feynman amplitude'
associated to $(\Gamma, \ell)$. By definition, $I_{\ell}(\Gamma)$
is obtained by the following rule: To each edge with end labels
$j,k$ one assigns a factor of $i h^{jk}$ where $H^{-1} =
(h^{jk}).$ To each closed vertex one assigns a factor of $i
\frac{\partial^{\nu} S (0)}{\partial x^{i_1} \cdots \partial
x^{i_{\nu}}}$ where $\nu$ is the valency of the vertex and $i_1
\dots, i_{\nu}$ at the index labels of the edge ends incident on
the vertex. To the open vertex, one assigns  the factor
$\frac{\partial^{\nu} a(0)}{\partial x^{i_1} \dots \partial
x^{i_{\nu}}}$, where $\nu$ is its valence.   Then
$I_{\ell}(\Gamma)$ is the product of all these factors.  To the
empty graph one assigns the amplitude $1$.  In summing over
$(\Gamma, \ell)$ with a fixed graph $\Gamma$, one sums the product
of all the factors as the indices run over $\{1, \dots, n\}$.

We note that the power of $k$ in a given term with $V$ vertices
and $I$ edges equals $k^{-\chi_{\Gamma'}}$, where $\chi_{\Gamma'}
= V - I$ equals the Euler characteristic of the graph $\Gamma'$
defined to be $\Gamma$ minus the open vertex. We note that there
are only finitely many graphs for each $\chi$ because the valency
condition forces $I \geq 3/2 V.$ Thus, $V \leq 2 j, I \leq 3 j.$

\subsection{The stationary phase calculations of $I^{0,+}_{2r,\rho}$: The data $D^{2j+2}_{2\vec \gamma} f_{\pm}(0)$}
In this section we will repeatedly use different parts of Lemma
\ref{AMPPROPS} without quoting them.

We first claim that in the stationary phase expansion of $I_{2r,
\rho}^{0,+}$, the data $D^{2j+2}_{2\vec \gamma} f_{\pm}(0)$ appears first in the
$k^{-j}$ term . This is because any labelled graph $(\Gamma, \ell)
$ for which $I_{\ell}(\Gamma)$ contains the factor
$D^{2j+2}_{2\vec \gamma} f_{\pm}(0)$ must have a closed vertex of valency $\geq
2j+2$,  or the open vertex must have  valency $\geq 2j+1.$ The
minimal absolute Euler characteristic $|\chi(\Gamma ')|$ in the
first case is $j$. Since the Euler characteristic is calculated
after the open vertex is removed, the minimal absolute Euler
characteristic in the second case is $j+1$ (there must be at least
$j+1$ edges.) Hence the latter graphs do not have minimal absolute
Euler characteristic.  More precisely, we have:

\begin{prop} \label{Gamma1j+1} In the stationary phase expansion of $I_{2r, \rho}^{0, {+}}$, the
only labelled graph $(\Gamma, \ell)$ with $\chi(\Gamma')= V-I=- j$
and $I_{\ell} (\Gamma)$ containing  $D^{2j+2}_{2\vec \gamma} f_{\pm}(0)$ is given
by:
\begin{itemize}
\item ${\Gamma}_{1, j+1}\in \mathcal G_{1, j+1} $  (i.e. $ V = 1,
I = j+1)$. The graph $\Gamma_{1,j+1}$ has no open vertex, one
closed vertex and $j+1$ loops at the closed vertex. \item
 The only labels producing the desired data are those $\ell_p$, with $1\leq p \leq 2r$ fixed, which labels
all endpoints of all edges as $(i,p)$ where $1 \leq i \leq n-1$. (Notice the label $(i,p)$ corresponds to the variable $x_p^i$.)
\end{itemize}
In addition, the sum of the Feynman amplitudes corresponding to
the labelled graphs $(\Gamma_{1,j+1}, \ell_p)$ above, for a fixed
$p$, is
$$\sum_{\ell_p} I_{\ell_p}(\Gamma_{1,j+1})\equiv  (4 r L) L^{-(n-1)r} \acal(r)  i^{j+2}\sum_{|\vec \gamma_p|=j+1}\frac{(j+1)!}{\vec \gamma_p!} (\vec {h_+^{11,pp}})^{ \vec \gamma_p}\,
w_+(p)D^{2j+2}_{2\vec \gamma_p} f_{w_+(p)}(0) $$
where we neglect terms with  $\leq 2j + 1$ derivatives.

\end{prop}

\begin{proof}

We argued diagrammatically that the power $k^{-j}$ is the greatest power of $k$ in
which $D^{2j+2}_{2\vec \gamma} f_{\pm}(0)$ appears. We also showed that a labeled graph with Euler characteristic $-j$ which produces $D^{2j+2}_{2\vec \gamma} f_{\pm}(0)$ must have a closed vertex
of valency $\geq 2j+2$. Now it is clear that such graph must have only one closed vertex and $j+1$ loops. This proves the first part of the proposition.
The second part follows easily from Lemma \ref{AMPPROPS}.

Now let us determine $\sum_{\ell_p}I_{\ell_p}(\Gamma_{1,j+1})$ for
the labelled graphs $(\Gamma_{1,j+1}, \ell_p)$  above. We have
\Small\begin{equation}
\sum_{\ell_p}I_{\ell_p}(\Gamma_{1,j+1})\equiv (2 r L) L^{-(n-1)r}
\acal(r)\,  i^{j+2} \sum_{\gamma_p^1+
...+\gamma_p^{n-1}=j+1}\frac{(j+1)!}{\gamma_p^1 ! \dots
\gamma_p^{n-1} !} (h_+^{11,pp})^{\gamma_p^1} ...
(h_+^{(n-1)(n-1),pp})^{\gamma_p^{n-1}}D_{2 \vec {\gamma_p}}^{2j+2}
{\mathcal L}_+ (0).
\end{equation}
\normalsize
So by Lemma \ref{AMPPROPS} and using short-hand notations for multi-indices we get
$$\sum_{\ell_p}I_{\ell_p}(\Gamma_{1,j+1})\equiv  (4 r L) L^{-(n-1)r} \acal(r) i^{j+2}\sum_{|\vec \gamma_p|=j+1}\frac{(j+1)!}{\vec \gamma_p!} (\vec {h_+^{11,pp}})^{ \vec \gamma_p}\,
w_+(p)D^{2j+2}_{2\vec \gamma_p} f_{w_+(p)}(0)$$
\end{proof}
\subsection{Proof of Theorem \ref{BGAMMAJ}}
Now we are ready to prove Theorem \ref{BGAMMAJ}. The discussion above shows that modulus derivatives of order $\leq 2j+1$
$$B_{\gamma^r, j}= (2rL)^{-1} L^{(n-1)r} \sum_{p=1}^{2r} \frac{\sum_{\ell_p}I_{\ell_p}(\Gamma_{1,j+1})}{S(\Gamma_{1,j+1})}.$$
We notice that
$S(\Gamma_{1,j+1})=|Aut(\Gamma_{1,j+1})|=2^{j+1}(j+1)!$. We then
break up the sums over $p$ of even/odd parity and use Proposition
\ref{HESSPAR} to replace the odd parity Hessian elements by
$h_+^{11}$ and the even ones by $h_+^{22}$. Taking into account
that $w_{+}(p) = 1 (-1)$ if $p$ is even (odd), we conclude that (
by the formula in Proposition \ref{Gamma1j+1})
$$B_{\gamma^r, j} =\frac{B_{\gamma^r,0}}{(2i)^{j+1}} \sum_{|\gamma|=j+1}\frac{r}{\vec \gamma !}\big\{(\overrightarrow {h^{11}_{+,
2r}})^{\vec \gamma}D^{2j+2}_{2\vec \gamma} f_{+}(0)
 -(\overrightarrow {h^{11}_{-,
2r}})^{\vec \gamma}D^{2j+2}_{2\vec \gamma} f_{-}(0)\big\}.$$
So far we have proved all parts of Theorem \ref{BGAMMAJ} except the last part which finds a formula for the wave invariants in the case of symmetries.
\subsubsection{Balian-Bloch invariants at bouncing ball orbits of
$({\Z}/{{2\Z}})^n$  symmetric domains} Now if we assume the $({\Z}/{{2\Z}})^n$
symmetry assumptions, namely $f_+=f=-f_-$ and $f$ being even in
all variables, then using (\ref{HPQ11}) the formula above
simplifies to \begin{equation} \label{WISYM} \begin{array}{l}
B_{\gamma^r, j} = \frac{B_{\gamma^r,0}}{(2i)^{j+1}} \sum_{|\vec
\gamma|=j+1}\frac{r}{\vec \gamma !}\left( \frac{1}{\sin \frac{\vec
\alpha}{2}} \cot \frac{r
\vec \alpha}{2} \right)^{\vec \gamma}D^{2j+2}_{2\vec \gamma} f(0)\\ \\
  +
\{\text{polynomial of Taylor coefficients of order} \leq
2j \}.
\end{array}\end{equation}
This finishes the proof of Theorem \ref{BGAMMAJ}. Q.E.D.

\subsection{\label{SIMPLE} Recovering the Taylor Coefficients and the Proof of Theorem \ref{ONESYM}} First of all we prove the following lemma
\begin{lem}
If $\{\al_1,\dots \al_{n-1}\}$ are linearly independent over $\mathbb Q$ then the functions
$$ \left(\cot \frac{r
\vec \alpha}{2} \right)^{\vec \gamma} $$ are linearly independent over $\mathbb C$
as functions of $r\in \N$.
\end{lem}
\begin{proof}
Suppose that there exist coefficients $c_{\vec \gamma}$ such that
$$\sum_{\vec \gamma}\; c_{\vec \gamma}\left(\cot \frac{r
\vec \alpha}{2} \right)^{\vec \gamma} = 0, \;\; \forall r \in \mathbb N.$$
Consider the function
$$ \psi(z_1, \dots, z_{n-1}) : = \sum_{\vec \gamma}\; c_{\vec \gamma} \left(\cot \vec z
 \,\right)^{\vec \gamma}.$$
This function is meromorphic and periodic of period $2 \pi$ in
each variable $z_j$, so it may be regarded as a meromorphic
function on $(\C/ \Z)^{n-1}$. It  vanishes when $z_j = r
\alpha_j/2 $ modulo $2 \pi$ for all $r = 1, 2, 3, \dots$. But such
points are dense in the real submanifold $(\R/\Z)^{n-1}$ and hence
the function vanishes identically on $(\C/ \Z)^{n-1}$. This is a
contradiction since the functions $\prod_{j = 1}^{n-1}
w_j^{\gamma_j}$ are independent functions and by the change of
variables $w_j = \cot z_j$ the functions $\prod_{j = 1}^{n-1}
\left(
 \cot z_j
\right)^{\gamma_j}$ must also be independent.
\end{proof}

Now assume $\Omega \subset \mathbb R^n$ is a domain in the class
$\mathcal D_{L}$ defined in (\ref{DL}). Take a non-degenerate
bounding ball orbit $\gamma$ of length $2L$ which satisfies all
the properties listed in (\ref{DL}). We would like to use
mathematical induction and recover the Taylor coefficients of the
function $f$ where $f$ and $-f$ are the local defining functions
of $\partial \Omega$ near the top and bottom of the bouncing ball
orbit respectively. First, it is possible to recover all the
$\al_j$, $1\leq j\leq n-1$, under a permutation \cite{Fr}. This is
because $|\det(I-P_{\gamma^r})|$ is a spectral invariant (the
$0$-th wave invariant). But we know that
$|\det(I-P_{\gamma^r})|=\prod_{j=1}^{n-1} (2-2\cos(r\al_j))$.
Hence $\prod_{j=1}^{n-1} (\sin^2(r\al_j/2))$ is a spectral
invariant for all $r\in \N$. It is easy to see that this condition
determines $\al_j$ under a permutation. We fix this permutation
and we argue inductively to recover all the Taylor coefficients.
Since $f$ is even in all the variables, the odd order Taylor
coefficients are zero. Now assume $D^{2 |\vec \gamma|}_{2\vec
\gamma} f(0)$ are given for all $|\vec \gamma|\leq j$. Hence the
remainder polynomial term of $(\ref{WISYM})$ is given. Now by the
above lemma, since all the functions $\left(\cot \frac{r \vec
\alpha}{2} \right)^{\vec \gamma}$ are linearly independent, we can
recover the Taylor coefficients $D^{2j+2}_{2\vec \gamma} f(0)$.
This concludes the proof of Theorem \ref{ONESYM}.

The analogous arguments will follow in the hyperbolic or mixed
hyperbolic-elliptic cases.


\end{document}